\documentclass[a4paper, 11pt]{article}

\usepackage[T1]{fontenc}
\usepackage[utf8]{inputenc}
\usepackage[english]{babel}	
\usepackage{lmodern}
\usepackage[dvipsnames]{xcolor}

\usepackage{amssymb}
\usepackage{amsthm}
\usepackage{amsmath}
\usepackage{mathtools}
\usepackage{mathrsfs}
\usepackage{bm}				

\usepackage[shortlabels]{enumitem}
\usepackage{hyperref}
\usepackage[bottom=2cm,top=2cm,left=2cm,right=2cm]{geometry}

\usepackage[autostyle]{csquotes}
\usepackage[style=alphabetic,doi=false,url=false,%
	isbn=false,%
    giveninits=true, maxbibnames=99, backend=biber]{biblatex}
    
\usepackage[capitalise, noabbrev]{cleveref}
	\theoremstyle{plain}
		\newtheorem{theorem}{Theorem}[section]
		\newtheorem{corollary}[theorem]{Corollary}
		\newtheorem{lemma}[theorem]{Lemma}
		\newtheorem{proposition}[theorem]{Proposition}
	\theoremstyle{definition}
		\newtheorem{definition}[theorem]{Definition}
		\newtheorem{notation}[theorem]{Notation}
		\newtheorem{assumption}[theorem]{Assumption}
	\theoremstyle{remark}
		\newtheorem{remark}[theorem]{Remark}

\usepackage{comment}
\renewcommand{\epsilon}{\varepsilon}

\DeclareMathOperator{\Lip}{Lip}					
\DeclareMathOperator{\lip}{lip}					
\DeclareMathOperator{\spt}{spt}					
\DeclareMathOperator{\Per}{Per}					
\DeclareMathOperator{\ess}{ess}					

\DeclarePairedDelimiter{\abs}{\lvert}{\rvert}	
\DeclarePairedDelimiter{\norm}{\lVert}{\rVert}	
\DeclarePairedDelimiter{\pths}{(}{)}			
\DeclarePairedDelimiter{\bkts}{[}{]}			
\DeclarePairedDelimiter{\brcs}{\lbrace}{\rbrace} 	

\newcommand{\stset}{\;\middle|\;}				

\newcommand{\numberset}{\mathbb}
\newcommand{\N}{\numberset{N}}

\newcommand{\R}{\numberset{R}}

\newcommand{\de}{\mathrm{d}}
\newcommand{\Xmms}{\mathsf{X}}					
	\newcommand{\dist}{\mathsf{d}}				
	\newcommand{\meas}{\mathfrak{m}}			
	\newcommand{\dmeas}{\,\de\meas}				
\newcommand{\Xdm}{(\Xmms,\dist,\meas)}			
\newcommand{\Odm}{(\Omega, \dist, \meas)}		

\newcommand{\Ymms}{\mathsf{Y}}					
	\newcommand{\distY}{\mathsf{d}_{\Ymms}}			
	\newcommand{\measY}{\mathfrak{m}_{\Ymms}}
	\newcommand{\Ydm}{(\Ymms,\distY,\measY)}

\newcommand{\Zmms}{\mathsf{Z}}					
	\newcommand{\distZ}{\mathsf{d}_{\Zmms}}			
	\newcommand{\measZ}{\mathfrak{m}_{\Zmms}}
	\newcommand{\Zdm}{(\Zmms,\distZ,\measZ)}

\newcommand{\XX}{\mathcal{X}}					
\newcommand{\Xdxm}{(\Xmms,\dist,x,\meas)}		
\newcommand{\Xdxmi}%
		{(\Xmms_i,\dist_i,x_i,\meas_i)}			

\DeclareMathOperator{\CD}{CD}
\DeclareMathOperator{\RCD}{RCD}
\newcommand{\mGH}{\mathrm{mGH}}					
\newcommand{\pmGH}{\mathrm{pmGH}}				
\newcommand{\pmG}{\mathrm{pmG}}					
\newcommand{\dmGH}{\dist_{\mGH}}
\makeatletter
\def\mwug{\@ifstar\@mwug\@@mwug}
	\def\@mwug#1{\abs*{\nabla #1}_{\mathrm{w}}}
	\def\@@mwug#1{\abs*{\nabla #1}}
\makeatother
\newcommand{\ball}[2]{B_{#2}\pths*{#1}}	

\newcommand{\hKN}{h_{K,N}}						
\newcommand{\HKN}{H_{K,N}}						
\newcommand{\cKN}{c_{K,N}}						
\newcommand{\mKN}{\meas_{K,N}}					
\newcommand{\dmKN}{\, \de \meas_{K,N}}
\newcommand{\JKN}{J_{K,N}}						
\newcommand{\DKN}{D_{K,N}}						
\newcommand{\isoKN}{\mcI_{K,N}}					
\newcommand{\lapKN}{\Delta_{K,N}}				
\newcommand{\deu}{\dist_{\mathrm{eu}}}			
\newcommand{\modelspace}{(\JKN, \deu, \mKN)}

\newcommand{\hNN}{h_{N-1,N}}					
	\newcommand{\HNN}{H_{N-1,N}}
	\newcommand{\mNN}{\meas_{N-1,N}}
	\newcommand{\lapNN}{\Delta_{N-1,N}}
	\newcommand{\dmNN}{\, \de \meas_{N-1,N}}
	\newcommand{\JNN}{J_{N-1,N}}

\newcommand{\leb}{\mathscr{L}}					
\newcommand{\borel}{\mathscr{B}}				
\newcommand{\C}{\mathcal{C}}					
\newcommand{\Cb}{\C_{\mathrm{b}}}				
\newcommand{\Eform}{\mathcal{E}}				
\newcommand{\Eop}{\mathcal{L}_{\Eform}}			
\newcommand{\Ch}{\mathrm{Ch}}					
\newcommand{\limi}{\lim_{i \to \infty}}
\newcommand{\ii}{\bm{i}}
\newcommand{\Wp}{W^{1,p}}						
\newcommand{\W}{W^{1,2}}						

\newcommand\restr[2]{{
  \left.\kern-\nulldelimiterspace 
  #1 
  \vphantom{\big|} 
  \right|_{#2} 
  }}
  
\newcommand{\PSz}{P\'{o}lya-Szeg\H{o}}
\newcommand{\Holder}{H\"{o}lder}

\newcommand{\proofstep}[1]{\textsc{#1}}

\let\latexchi\chi
\makeatletter
\renewcommand\chi{\@ifnextchar_\sub@chi\latexchi}
\newcommand{\sub@chi}[2]{
  \@ifnextchar^{\subsup@chi{#2}}{\latexchi^{}_{#2}}%
}
\newcommand{\subsup@chi}[3]{
  \latexchi_{#1}^{#3}%
}
\makeatother

\newcommand{\ie}{\emph{i.e.}}
\newcommand{\wrt}{with respect to}

\newcommand{\mms}{metric measure space}
\newcommand{\pmms}{pointed metric measure space}

\def\MySet#1{\expandafter\def\csname mc#1\endcsname{\mathcal{#1}}%
			 \expandafter\def\csname ms#1\endcsname{\mathscr{#1}}%
			 }
\def\ALLvec#1{\ifx#1\ALLvec\else\MySet#1\expandafter\ALLvec\fi}
\ALLvec ABCDEFGHIJKLMNOPQRSTUVWXYZ \ALLvec

\title{A Talenti-type comparison theorem for $\RCD(K,N)$ spaces and applications}
\author{	Andrea Mondino%
				\thanks{%
					Andrea Mondino: Mathematical Institute, University of Oxford, UK.\newline
					Email: \href{mailto:Andrea.Mondino@maths.ox.ac.uk}{Andrea.Mondino@maths.ox.ac.uk}}%
			\and Mattia Vedovato%
				\thanks{%
					Mattia Vedovato: Dipartimento di Matematica, Università degli Studi di Trento, Trento, Italy.\newline
					Email: \href{mailto:mattia.vedovato@unitn.it}{mattia.vedovato@unitn.it}}} 

\date{\today}

\makeindex

\begin{document}
\maketitle

\abstract 
We prove pointwise and $L^{p}$-gradient comparison results for solutions to  elliptic Dirichlet problems defined on open subsets of a (possibly non-smooth) space with positive Ricci curvature (more precisely of an $\RCD(K,N)$ metric measure space, with $K>0$ and $N\in (1,\infty)$). The obtained Talenti-type comparison is sharp, rigid and stable with respect to $L^{2}$/measured-Gromov-Hausdorff topology; moreover, several aspects seem new even for smooth Riemannian manifolds. As applications of such Talenti-type comparison, we prove  a  series of improved Sobolev-type inequalities, and an  $\RCD$ version of the St.~Venant-P\'olya torsional rigidity comparison theorem (with associated rigidity and stability statements). Finally, we give a  probabilistic interpretation (in the setting of smooth Riemannian manifolds) of the aforementioned comparison results,  in terms of exit time from an open subset for the Brownian motion.
\tableofcontents
\bigskip


\section{Introduction}
In the study of geometric and variational problems in Euclidean spaces, a tool which often proves useful is the technique of \emph{symmetrization}: one can frequently simplify a complex problem by reducing it to the study of spherically symmetric objects. 
Specifically, the classical notion of \emph{Schwarz symmetrization} of a function plays a notable role in proving results such as the Rayleigh-Faber-Krahn Inequality, as well as several variational inequalities for differential boundary problems. The rough idea is the following: 
for any bounded measurable domain $\Omega\subset \R^n$, one considers the unique ball $\Omega^\star\subset \R^n$ centered at the origin and having the same volume of $\Omega$; then, given a measurable function $u:\Omega\to[0,\infty)$, one constructs a ``symmetrized'' function $u^\star:\Omega^\star\to[0,\infty)$ which is radial, decreases in the radial variable, and its super-level sets $\{u^{\star}>t\}$ have the same Lebesgue measure as the corresponding super-level sets $\{u>t\}$ of $u$.

The idea of using symmetrizations to infer comparison results for elliptic boundary value problems goes back (at least) to the proofs by Faber \cite{Faber1923} and Krahn \cite{Krahn1925} of  Lord Rayleigh's conjecture about the principal frequency for an elastic membrane,  and to the work of Szeg\H{o} \cite{Szego1950, Szego1958} on the clamped and buckling plate problems. Estimates on solutions to differential boundary value problems via Schwarz symmetrization have been then obtained by several mathematicians, let us mention Weinberger \cite{Weinberger1962}, Bandle \cite{Bandle1976}, Talenti \cite{Talenti1976}, P. L. Lions \cite{Lions1981}, Alvino-Lions-Trombetti \cite{ALT1990}. The corresponding paradigmatic result is now well known in the literature as ``Talenti comparison theorem'', and we keep such terminology.
\par The basic idea is to compare the outcomes of the following two procedures:
	\begin{enumerate}[(a)]
	\item Solve a  Poisson problem of the type
		\begin{equation}\label{eq:PoissonIntro}
		\left\lbrace
			\begin{aligned}
			-\Delta u 	&= 	f 	&&\text{in $\Omega\subset\R^n$}\\
			u			&=	0	&&\text{on $\partial\Omega$}
			\end{aligned}
		\right.,
		\end{equation}
	with $f\in L^2(\Omega)$; then consider the Schwarz symmetrization $u^\star$ of $u$.
	\item Solve the symmetrized Poisson problem
		\begin{equation}
		\left\lbrace
			\begin{aligned}
			-\Delta v 	&= 	f^\star 	&&\text{in $\Omega^\star\subset\R^n$}\\
			v			&=	0			&&\text{on $\partial\Omega^\star$}
			\end{aligned}
		\right..
		\end{equation}
	\end{enumerate} 
Talenti  \cite{Talenti1976}, sharpening the aforementioned  \cite{Weinberger1962, Bandle1976}, proved  that the pointwise inequality $u^\star \leq v$ holds $\leb^n$-almost everywhere in $\Omega^\star$; moreover, if $u^\star=v$ holds almost everywhere in $\Omega^\star$, then $\Omega$ itself was already a ball.
We refer the reader to \cite{Bandle1980, Baernstein2019, Lions1981, Kesavan2006, PolyaSzego1951} for different proofs and related topics.

\medskip

The aim of the present work  is to generalize such a comparison result to a curved, possibly non-smooth, setting.
\\The framework of the paper is the one of metric measure spaces with Ricci curvature bounded below (by a constant $K>0$) and dimension bounded above (by $N\in (1,\infty)$) in a synthetic sense, via optimal transport. Recall that a metric measure space is a triplet $(\Xmms,\dist,\meas)$ where $(\Xmms, \dist)$ is a complete separable metric space, and $\meas$ is a non-negative Borel measure finite on bounded sets. 
\\More precisely, the paper will be in the framework of $\RCD(K,N)$ spaces, with $K>0$ and $N\in (1,\infty)$. We refer the reader to  \cref{subsec:cdbounds} for more details about the definition and the relevant literature; for the sake of this introduction, we only mention that the class of $\RCD(K,N)$ spaces includes as remarkable examples:
\begin{itemize}
\item Riemannian manifolds with Ricci curvature $\geq K$ and dimension $\leq N$, as well as their measured-Gromov-Haudorff limits;
\item Alexandrov spaces with Hausdorff dimension $\leq N$ and curvature $\geq K/(N-1)$.
\end{itemize}
Moreover, if $K>0$, a generalized version of the Bonnet-Myers theorem implies that $\spt(\meas)$ is compact and thus $\meas(\Xmms)<\infty$. Up to a constant normalization of the measure, we can thus assume that $\meas(\Xmms)=1$ (see \cref{rem:scaling}).
\\Let us stress that the results of the paper seems new even for smooth Riemannian manifolds with Ricci curvature bounded below by a positive constant. 

It is worth to mention that, while in the Euclidean setting the Schwarz symmetrization is defined on balls \emph{in the very same Euclidean space}, for the curved setting of an $\RCD(K,N)$ space the symmetrization is defined on a ``model space'' depending only on $K>0$ and $N\in (1,\infty)$. Such a model space is given by  an interval $\JKN$ of the real line, endowed with the Euclidean distance and a measure which is absolutely continuous \wrt{} the Lebesgue measure:
	\begin{equation}\label{eq:ModelSpaceIntro}
	\JKN \doteq \bkts*{0,\pi\sqrt{\tfrac{N-1}{K}}}, \qquad \hKN(t) \doteq \tfrac{1}{c_{K,N}} \sin^{N-1}\pths*{t \sqrt{\tfrac{K}{N-1}}}, \qquad \mKN \doteq \hKN \leb^1,
	\end{equation}
where $c_{K,N}$ is a normalizing constant. Observe that, when $N\geq 2$ is an integer, such a model space naturally corresponds to a round sphere of dimension $N$ and constant Ricci curvature $K$ (see \cref{rem:Sphere}).

The main result of the paper is  a Talenti-type comparison theorem where we compare the weak solution to a  Poisson problem as in \eqref{eq:PoissonIntro} defined on an open set $\Omega$ of an $\RCD(K,N)$ space ($K>0$, $N\in (1,\infty)$) with the solution of an analogous Poisson problem defined on the model space \eqref{eq:ModelSpaceIntro} (see \cref{theorem:talenti}). Actually, we consider more generally any second order elliptic operator arising as infinitesimal generator of any strongly local, uniformly elliptic bilinear form (see \cref{sec:comparison}).
The comparison is (trivially) sharp as equality is attained in the model space \eqref{eq:ModelSpaceIntro}, which is $\RCD(K,N)$.

We will also establish:
\begin{itemize}
\item A \emph{rigidity} result (\cref{thm:talenti_rigidity}) roughly stating that if equality in the Talenti-type comparison \cref{theorem:talenti} is achieved, then the space is a spherical suspension;
\item A \emph{stability} result (\cref{thm:stabilityTalenti})  roughly stating that equality in the Talenti-type comparison \cref{theorem:talenti} is almost achieved (in $L^{2}$-sense) if and only if the space is mGH-close to a spherical suspension.
\end{itemize}
Finally, as applications of the Talenti-type comparison \cref{theorem:talenti}, we will establish:
 \begin{itemize}
 \item A  series of Sobolev-type inequalities that to best of our knowledge are new in the framework of $\RCD(K,N)$ spaces (\cref{Cor:SobolevIneq});
 \item An $\RCD(K,N)$ version of the St.~Venant-P\'olya torsional rigidity comparison theorem, with associated rigidity and stability statements (\cref{thm:StVenPol}). 
 \item A probabilistic interpretation (in the setting of smooth Riemannian manifolds) of the comparison results obtained in the paper,  in terms of exit time from an open subset for the Brownian motion  (\cref{cor:ET}).
 \end{itemize}
 For the reader's convenience, the Appendix gives a self-contained presentation of the statements of the main results for a smooth Riemannian manifold with positive Ricci curvature (as several aspects seem to be new even in this setting).
 
 \subsection*{Acknowledgements}
A.M. is supported by the European Research Council (ERC), under the European's Union Horizon 2020 research and innovation programme, via the ERC Starting Grant  “CURVATURE”, grant agreement No. 802689.


\section{Preliminaries}
\label{sec:prelim}


\subsection{Perimeters, isoperimetric profiles and Sobolev spaces}
\label{subsec:perim} 
Throughout the paper, $\Xdm$ is a metric measure space with $(\Xmms,\dist)$ a complete and separable metric space (actually, as a consequence of the assumptions of our main theorems, $(\Xmms,\dist)$ will be compact) and $\meas$ is a Borel probability measure on $\Xmms$ with $\spt(\meas)=\Xmms$ (see \cref{rem:scaling} below for a justification of such standing working assumptions).
We start by recalling the notion of slope of a real-valued function.

\begin{definition}[Slope]
Let $(\Xmms,\dist)$ be a metric space and $u:\Xmms\to \R$ be a real valued function. We define the \emph{slope} of $u$ at the point $x\in\Xmms$ as
	\begin{equation*}
	(\lip u)(x)\doteq 
		\begin{cases}
			\limsup_{y\to x} \frac{\abs*{u(x)-u(y)}}{d(x,y)} 	&\text{if $x$ is not isolated}		\\
			0		&\text{otherwise}.
		\end{cases}
	\end{equation*}
\end{definition}
From now on $\Lip (\Xmms)=\Lip(\Xmms,\dist)$ will denote the space of Lipschitz maps on $(\Xmms,\dist)$, while $\Lip_c (\Xmms)=\Lip_c(\Xmms,\dist)$ will be the subspace of compactly supported Lipschitz maps.

Given a metric measure space $\Xdm$, one can introduce a notion of \emph{perimeter} which extends the classical one on $\R^n$. The following definition was first introduced in \cite{Miranda2003} and further explored in \cite{Ambrosio2001, Ambrosio2002}:

\begin{definition}[Perimeter]
Let $E\in\borel (\Xmms)$, where $\borel(\Xmms)$ denotes the class of Borel sets of $(\Xmms,\dist)$, and let $A\subset \Xmms$ be open. We define the perimeter of $E$ relative to $A$ as:
	\begin{equation*}
	\Per (E;A)\doteq \inf \brcs*{\liminf_{n\to \infty} \int_A \lip u_n \dmeas \stset u_n\in \Lip(A),\,\, u_n\to \chi_E \text{ in $L^1(A,\meas)$}},
	\end{equation*}
where $ (\lip u_n)(x)$ is the slope of $u$ at the point $x$.

If $\Per (E;\Xmms)<\infty$, we say that $E$ is a \emph{set of finite perimeter}.
\end{definition}

When $E$ is a fixed set of finite perimeter, the map $A\mapsto \Per (E;A)$ is the restriction to open sets of a finite Borel measure on $\Xmms$, defined as
	\begin{equation*}
		\Per(E;B)\doteq \inf \brcs*{\Per(E;A) \stset \text{$A$ open}, A\supset B}.
	\end{equation*}

\begin{definition}[Isoperimetric profile]\label{def:iso_profile}
Let $(\Xmms,\dist,\meas)$ be a \mms{} with $\meas(\Xmms)=1$. The \emph{isoperimetric profile} $\mcI=\mcI_{\Xdm}:[0,1]\to[0,+\infty)$ is defined as 
	\begin{equation}\label{eq:iso_profile}
	\mcI_{\Xdm}(v)\doteq \inf\brcs*{\Per(E)\stset E\in\msB(\Xmms), \meas(E)=v}, \quad v\in [0,1].
	\end{equation}
\end{definition}

Finally, we recall the notion of Cheeger energy of an $L^p$ function, which will be used to define  Sobolev spaces on \mms{}s. For a review of this theory, we refer the reader to \cite{Ambrosio2014c, Ambrosio}, as well as the pioneering work \cite{Cheeger1999}. As  shown in those references, the definition of $\Wp$ through the Cheeger energy is not the only approach available, but one can prove that other relevant strategies (e.g. Newtonian spaces) turn out to be equivalent in the framework of this paper.

\begin{definition}[Cheeger energy]\label{def:cheeger_energy}
Let $\Xdm$ be a metric measure space, let $p\in(1,+\infty)$ and let $f\in L^p(\Xmms, \meas)$. The \emph{$p$-Cheeger energy} of $f$ is defined as
	\begin{equation}\label{eq:cheeger_energy}
	\Ch_p(f)\doteq \inf \brcs*{\liminf_{n\to \infty}\frac{1}{p}\int \pths*{\lip f_n}^p \dmeas \stset 
		\begin{gathered}
		f_n\in \Lip(\Xmms)\cap L^p(\Xmms,\meas)\\
		\norm*{f_n-f}_{L^p}\to 0
		\end{gathered}},
	\end{equation}
where $(\lip f_n)(x)$ is the slope of $f_n$ at the point $x$.
\end{definition}

\begin{definition}[Sobolev spaces]\label{def:sob_spaces}
Given $\Xdm$ \mms{},  $p\in (1,+\infty)$, and an open subset $\Omega\subset \Xmms$, we define:
	\begin{enumerate}
	\item The space $\Wp \Xdm$ as the space of functions $f\in L^p(\Xmms, \meas)$ with finite $p$-Cheeger energy, endowed with the norm
		\begin{equation*}
		\norm*{f}_{\Wp \Xdm} \doteq  \brcs*{\norm*{f}_{L^p(\Xmms,\meas)}+p\, \Ch_p(f)}^{\frac{1}{p}}
		\end{equation*}
		which makes $\Wp \Xdm$ a Banach space.
	\item The space $\Wp_0(\Omega)$ as the closure of $\Lip_c(\Omega)$ with respect to the norm of $\Wp \Xdm$.
	\end{enumerate}
\end{definition}

For any $f\in \Wp \Xdm$, one can single out a distinguished object $\mwug*{f}\in L^p (\Xmms, \meas)$, which plays the role of the modulus of the gradient and provides the integral representation
	\begin{equation*}
	\Ch_p(f)=\frac{1}{p} \int \mwug*{f}^p \dmeas;
	\end{equation*}
this function is called the \emph{minimal $p$-weak upper gradient}  of $f$ and can be obtained through an optimal approximation in 
\cref{eq:cheeger_energy}; in order to keep the notation clearer, from now on we will omit the subscript $\mathrm{w}$ and simply denote the minimal weak upper gradient of $f$ as $\mwug{f}$. We refer the reader to \cite{Heinonen2015,BjornBjorn, Cheeger1999, Ambrosio2014c} for details.  
A priori the minimal $p$-weak upper gradient may depend on $p$; however in locally doubling and Poincar\'e spaces (and $\RCD(K,N)$ spaces are so, see \cite{Sturm2006b, Rajala2012}) it is independent of $p$ by the deep work of Cheeger \cite{Cheeger1999}.

We now introduce a \emph{local} notion of Sobolev space, which will be needed in \cref{subsec:alm_rig}, and relies on the definition given in \cite[Definition 2.14]{AmbrosioHonda}. We specialize to the case $p=2$, which is the only one we will use.

\begin{definition}[Local Sobolev space]\label{def:loc_sob_spaces}
Let $\Xdm$ be a \mms{} and let $\Omega \subset \Xmms$ be an open subset. We say that $f\in L^2 (\Omega,\meas)$ belongs to $\W(\Omega, \dist, \meas)$ if 
	\begin{enumerate}[(a)]
	\item\label{cond:loc_sob1}  
		for any $\phi\in \Lip_c (\Xmms, \dist)$  with $\spt(\phi)\subset \Omega$, it holds $\phi f \in \W \Xdm$  (where $\W \Xdm$ is the global Sobolev space defined in \cref{def:sob_spaces});
	\item\label{cond:loc_sob2}  
		$\mwug{f} \in L^2(\Omega, \meas)$.
	\end{enumerate}
\end{definition}

Notice that the property \ref{cond:loc_sob1}, together with the locality properties of the minimal weak upper gradient, guarantees that the condition in \ref{cond:loc_sob2} is well posed (see again \cite{AmbrosioHonda}).


\subsection{Curvature-dimension bounds and infinitesimal Hilbertianity}
\label{subsec:cdbounds}

All the results in the paper will be in the framework of $\RCD(K,N)$ spaces.  We recall here very briefly and schematically the main definitions involved   
(for more details see the original papers \cite{LottVillani, Sturm2006a, Sturm2006b,Ambrosio2014b, AGMR2015, Gigli2015, EKS2015, AMS2019, CaMi2016}, or  \cite{Ambrosio} for a survey on the subject). 
In what follows, $\Xdm$  will be a complete and separable \mms{}, and $K, N$ will be real numbers with $N\in(1,\infty)$.
\begin{itemize}
	\item For any metric space $(\Ymms,\dist_{\Ymms})$, we denote by $\msP(\Ymms)$ the space of Borel probability measures on $\Ymms$, and by $\msP_2(\Ymms)$ the space of Borel probability measures with finite second moment.
	\item The \emph{Wasserstein distance} $W_2$ on $\msP_2(\Xmms)$ is defined as
		\begin{equation}\label{eq:wasser_dist}
		W_2(\mu_0,\mu_1) \doteq 
			\inf \brcs*{ \int_{\Xmms \times \Xmms} \dist^2 (x,y) \, \de \gamma (x,y)
						\stset 
						\gamma \in \msP (\Xmms \times \Xmms), 
						\pi^{(0)}_\sharp \gamma = \mu_0,
						\pi^{(1)}_\sharp \gamma = \mu_1},
		\end{equation}
		for any $\mu_0,\mu_1 \in \msP_2(\Xmms)$, where $\pi^{(0)}$ is the projection on the first component, $\pi^{(1)}$ is the projection on the second component, and the subscript $\sharp$ indicates the pushforward of the measure.
	\item The space $\operatorname{Geo}(\Xmms)$ is the space of constant speed geodesics on $\Xmms$:
		\begin{equation*}
		\operatorname{Geo}(\Xmms)
			\doteq \brcs*{
			\gamma \in \C([0,1],\Xmms) \stset 
			\text{$\dist(\gamma(s),\gamma(t))=\abs{s-t}\dist(\gamma(0),\gamma(1))$ for any $s,t\in[0,1]$}}.
		\end{equation*}
		For any $t\in[0,1]$, the \emph{evaluation map} ${\rm e}_t$ is defined on $\operatorname{Geo}(\Xmms)$ as
		\begin{equation*}
		{\rm e}_t(\gamma)\doteq \gamma(t)\quad\text{for any $\gamma \in \operatorname{Geo}(\Xmms)$}.
		\end{equation*}
	\item For any pair of measures $\mu_0$, $\mu_1$ in $\msP_2(\Xmms)$, the set of \emph{dynamical optimal plans} are defined as
		\begin{equation*}
		\operatorname{OptGeo}(\mu_0,\mu_1)
		\doteq 	\brcs*{\nu \in \msP(\operatorname{Geo}(\Xmms))\stset
			\text{$({\rm e}_0,{\rm e}_1)_\sharp \nu$ realizes the minimum in \cref{eq:wasser_dist}}}.
		\end{equation*}
	\item For any $\theta >0$ and $t\in[0,1]$, the \emph{distortion coefficients} are defined as
		\begin{equation*}
		\tau^{(t)}_{K,N}(\theta) \doteq t^{\frac{1}{N}} \sigma^{(t)}_{K,N}(\theta)^{\frac{N-1}{N}},
		\end{equation*}
		where
		\begin{equation*}
		\sigma^{(t)}_{K,N}(\theta)\doteq 
			\begin{cases}
				\infty
					&\text{if $K\theta^2\geq N\pi^2$}\\
				\frac{\sin(t\theta \sqrt{K/N})}{\sin(\theta \sqrt{K/N})}
					&\text{if $0<K\theta^2< N\pi^2$}\\
				t
					&\text{if $K\theta^2<0$ and $N=0$, or if $K\theta^2=0$}\\
				\frac{\sinh(t\theta \sqrt{K/N})}{\sinh(\theta \sqrt{K/N})}
					&\text{if $K\theta^2\leq0$ and $N>0$}
			\end{cases}.
		\end{equation*}
	\item The \emph{Rényi entropy} functional $\mcE_N:\msP (\Xmms) \to [0,\infty]$ is defined as
		\begin{equation*}
		\mcE_N(\mu)\doteq \int_{\Xmms} \rho^{1-\frac{1}{N}} \dmeas, \quad \text{where $\mu = \rho\meas+\mu^s$ and $\mu^s\perp \meas$}.
		\end{equation*}
	\item \emph{$\CD$ condition}: we say that $\Xdm$ verifies the $\CD(K,N)$ condition for some $K\in \R$, $N\in(1,\infty)$ if: for any pair of probability measures  
	$\mu_0,\mu_1\in\msP(\Xmms) $ with bounded support and with $\mu_0,\mu_1\ll \meas$, there exists $\nu \in \operatorname{OptGeo}(\mu_0,\mu_1)$ and an optimal plan $\pi\in \msP(\Xmms\times\Xmms)$ such that $\mu_t\doteq ({\rm e}_t)_\sharp \nu \ll \meas$ and 
		\begin{equation*}
		\mcE_{N'}(\mu_t)\geq 
			\int\bkts*{
				\tau^{(1-t)}_{K,N'}(\dist(x,y))\rho_0^{-\tfrac{1}{N'}}+
				\tau^{(t)}_{K,N'}(\dist(x,y))  \rho_1^{-\tfrac{1}{N'}} 
				} \de \pi (x,y)
		\end{equation*}
		for any $N'\geq N$, $t\in[0,1]$.
	\item We say that $\Xdm$ is \emph{infinitesimally Hilbertian} if the Cheeger energy $\Ch_2$ defined in \eqref{eq:cheeger_energy} is a quadratic form  
	on $\W \Xdm$. In that case, we still denote by $\Ch$ the symmetric bilinear form associated to $\Ch=\Ch_2$.
	\item We say that $\Xdm$ satisfies the \emph{$\RCD(K,N)$ condition} if it satisfies the $\CD(K,N)$ condition and it is infinitesimally Hilbertian.
\end{itemize}

\begin{remark}[Scaling properties and standard normalizations]\label{rem:scaling}
One can define the $\RCD(K,N)$ condition for a complete and separable metric space endowed with a non-negative Borel measure which is finite on bounded subsets.  From the very definitions, it is not difficult to check that for any $\lambda$ and $c >0$ the following implication holds
	\begin{equation}\label{eq:Scaling}
	\text{$\Xdm$ is an $\RCD(K,N)$ space} \quad \Longrightarrow \quad \text{$(\Xmms,\lambda \dist, c \meas)$ is an $\RCD\pths*{ \lambda^{-2}K,N}$ space}.
	\end{equation}
If $K>0$, the Bonnet-Myers Theorem (proved for $\CD(K,N)$ spaces in \cite{Sturm2006b}) implies that $(\Xmms,\dist)$ is compact with $\meas(\Xmms)\in (0,\infty)$. Thanks to the scaling property	\eqref{eq:Scaling}, up to  constant scalings, it is not restrictive to assume $\meas(\Xmms)=1$ and $K=N-1$.
\end{remark}


\subsection{1-dimensional model spaces}
\label{subsec:model_spaces}

In this \namecref{subsec:model_spaces} we recall the 1-dimensional  ``model'' \mms{}s with Ricci curvature bounded below by $K>0$ and dimension bounded above by $N\in (1,\infty)$  singled out in \cite[Appendix C]{Gromov2007} and \cite{Milman2015} on which we will construct the needed symmetrizations. 
Let $K>0$ and $N\in (1,\infty)$. Let $\JKN$ be the interval
	\begin{equation*}
	\JKN 	\doteq \bkts*{0,\pi\sqrt{\tfrac{N-1}{K}}},
	\end{equation*} 
and define the following probability density function on $\JKN$:
	\begin{equation*}
	\hKN(t) \doteq \frac{1}{c_{K,N}} \sin^{N-1}\pths*{t \sqrt{\tfrac{K}{N-1}}},
	\end{equation*}
where $c_{K,N}$ is the normalizing constant
	\begin{equation*}
	c_{K,N}	\doteq \int_{\JKN} \sin^{N-1}\pths*{t \sqrt{\tfrac{K}{N-1}}} \,\de \leb^1(t).
	\end{equation*}
	
\begin{definition}[Model spaces]\label{def:model_space}
Let $K>0$ and $N\in (1,\infty)$. We define the one dimensional model space with curvature parameter $K$ and dimension parameter $N$ as $(\JKN, \dist_{\mathrm{eu}}, \mKN)$, where $\mKN \doteq \hKN \leb^1 \llcorner \JKN$ and $\dist_{\mathrm{eu}}$ is the standard euclidean distance.
\end{definition}

\begin{remark}\label{rem:Sphere}
When $N\in\N$, $\mKN([0,x])$ represents the measure of the geodesic ball of radius $x$ on the $N$-dimensional sphere of Ricci curvature $K$, endowed with the canonical metric. Notice however that \cref{def:model_space} makes sense when $N$ is not a natural number as well.
\end{remark}

\begin{notation}\label{notation:HKN}
For the sake of convenience, we  will also denote by $H_{K,N}$ the cumulative distribution function of $\meas_{K,N}$, \ie{}:
	\begin{equation*}
	H_{K,N}(x)\doteq \meas_{K,N}([0,x])=\int_0^x h_{K,N}(t)\,\de t.
	\end{equation*}
\end{notation}

The following \namecref{lem:asympt} is an elementary consequence of the definitions of $\hKN$ and $\HKN$:

\begin{lemma}\label{lem:asympt}
Let $K>0$ and $N\in (1,\infty)$ be fixed. Then:
	\begin{enumerate}
	\item If $\gamma_1(K,N)\doteq \frac{1}{\cKN}\pths*{\frac{K}{N-1}}^{\frac{N-1}{2}}$, then
		\begin{equation*}
		\lim_{t\to 0^+} \frac{\hKN(t)}{t^{N-1}}=\gamma_1(K,N),  \quad \text{and} \quad   \hKN(t)\leq \gamma_1(K,N) t^{N-1} \quad \forall t\in\JKN.
		\end{equation*}			  
  		Moreover, for any $r_{1}\in \left(0, \pi \sqrt{\frac{N-1}{K}} \right)$ there exists $C=C(r_{1},K,N)>0$ such that
		\begin{equation*}
		h_{K,N}(t)\geq C\, t^{N-1}, \quad \forall t\in (0,r_{1}).
		\end{equation*}		   
	\item $\HKN$ is invertible on $\JKN$; moreover, if $\gamma_2(K,N)\doteq \frac{\gamma_1(K,N)}{N}$:
	\begin{align*}
		\lim_{t\to 0^+} \frac{\HKN(t)}{t^N}&=\gamma_2(K,N),
		\quad &\text{and} \quad
		\HKN(t)&\leq \gamma_2(K,N) t^{N}  \quad \forall t\in\JKN;\\
		\lim_{t\to 0^+} \frac{\HKN^{-1}(t)}{t^{\frac{1}{N}}}&=\frac{1}{\gamma_2(K,N)^{\frac{1}{N}}},
		\quad &\text{and} \quad
		\HKN^{-1}(t)&\geq \frac{1}{\gamma_2(K,N)^{\frac{1}{N}}} t^{\frac{1}{N}} \quad \forall t\in(0,1).
	\end{align*}
	\end{enumerate}
\end{lemma}

For the model space $\JKN$ defined before, we can find an almost explicit expression for the isoperimetric profile.

\begin{lemma}[Isoperimetric profile of $\JKN$]
	\label{lemma:isoprof_model}
The isoperimetric profile $\mcI_{K,N}$ of the model space 
\begin{equation*}
(J_{K,N}, \dist_{\mathrm{eu}}, \meas_{K,N})
\end{equation*}
is given by the following formula:
	\begin{equation*}
	\mcI_{K,N}(v)=\hKN(\HKN^{-1}(v)), \quad v\in [0,1].
	\end{equation*}
Moreover, the $\inf$ in \cref{eq:iso_profile} is attained at the intervals
\begin{equation}\label{eq:isoper_interv}
(0,\HKN^{-1}(v)) \quad \text{and} \quad (\HKN^{-1}(1-v),\DKN),
\end{equation}
where $\DKN\doteq \pi \sqrt{\frac{N-1}{K}}$.
\end{lemma}

In other words: $\mcI_{K,N}(v)$ coincides with the density function computed at the point $x$ such that $\mKN([0,x])=v$.

\begin{proof}
The proof is a slight modification of  \cite{Bobkov1996}, we include it here for the reader's convenience.
Thanks to \cite[Proposition 3.1]{Cavalletti2018a}, we know that if $E$ has finite perimeter in $\JKN$, then it is $\mKN$-equivalent to a countable union of closed disjoint intervals, i.e. there exists a sequence of pairwise disjoint intervals $\brcs*{\bkts*{a_i, b_i}}_{i\in\N}$ such that $[a_i,b_i]\subset {\JKN}$ and
	\begin{equation}\label{eq:disj_intervals}
	\mKN\pths*{E\triangle \bigcup_{i \in\N}[a_i,b_i]}=0,
	\end{equation}
thus it suffices to consider such unions.
Moreover, by the same result, if \cref{eq:disj_intervals} holds then one has:
	\begin{equation*}
	\Per (E) = \sum_{i=0}^\infty \pths*{\hKN(a_i)+\hKN(b_i)}.
	\end{equation*}
	
\proofstep{Step 1:} We claim that the intervals in \cref{eq:isoper_interv} are minimal among the class of closed intervals.
Let $v\in(0,1)$; notice that the problem trivializes at $0$ and $1$. We denote by $f_v:(0,\HKN^{-1}(1-v))\to {\JKN}$ the function defined by
	\begin{equation*}
	f_v(x)\doteq \HKN^{-1}\pths*{\HKN(x)+v},
	\end{equation*}
that is: $f_v(x)$ is the unique element of ${\JKN}$ such that the interval $(x, f_v(x))$ has measure $v$. Notice that $\HKN(f_v(x))-\HKN(x)=v$, thus $\hKN(f_v(x))f_v'(x)=\hKN(x)$. 

Moreover, we denote by $p_v: (0,\HKN^{-1}(1-v))\to (0,+\infty)$ the function
	\begin{equation*}
	p_v(x)\doteq \Per \pths*{ (x, f_v(x))} = \hKN(x)+\hKN(f_v(x)).
	\end{equation*}
By differentiating with respect to $x$, one finds:
 	\begin{equation*}
 	p_v'(x)=\hKN'(x)+\hKN'(f_v(x))f_v'(x)=\hKN(x)\pths*{\frac{\hKN'(x)}{\hKN(x)}+\frac{\hKN'(f_v(x))}{\hKN(f_v(x))}}.
 	\end{equation*}
By easy computations, one can see that the map $z\mapsto \frac{\hKN'(z)}{\hKN(z)}$ coincides with 
	\begin{equation*}
	z \mapsto	{\sqrt{K(N-1)}} \cot \pths*{\sqrt{\tfrac{K}{N-1}}z},
	\end{equation*}
thus it is always decreasing; on the other hand, $f_v(\cdot)$ is strictly increasing. As a consequence, the map $x\to \frac{p_v'(x)}{\hKN(x)}$ is strictly decreasing; moreover, it tends to $+\infty$ when $x\downarrow 0$, while it tends to $-\infty$ when $x\uparrow \HKN^{-1}(1-v)$. This means there exists a value $x_v$ such that $p_v'>0$ on $(0,x_v)$ and $p_v'<0$ on $\pths[\big]{x_v, \HKN^{-1} (1-v)}$; noting that by symmetry $p_v(0)=p_v(\HKN^{-1}(1-v))$, we conclude that the sets in \cref{eq:isoper_interv} are minimal for the perimeter among intervals.

We also notice that 
	\begin{equation}\label{eq:x_v}
	x_v < \frac{\DKN}{2}
	\quad \text{and} \quad
	f(x_v)>\frac{\DKN}{2}
	\end{equation}
must hold. Indeed, by exploiting the symmetry of $\hKN$ (\ie{}, the fact that $\hKN(\DKN-x)=\hKN(x)$ for any $x\in\JKN$), it is easy to see that for any $x\in \pths*{0,\HKN^{-1}(1-v)}$ the identity 
	\begin{equation*}
	\mKN([\DKN-f_v(x),\DKN-x])=v
	\end{equation*}
holds, so that $f_v(\DKN-f_v(x))=\DKN-x$. As a consequence,
 	\begin{equation*}
 	\begin{split}
 	p_v(x) 	&= \hKN(\DKN-x) + \hKN(\DKN-f_v(x))\\
 			&= \hKN(\DKN-f_v(x))+\hKN(f_v(\DKN-f_v(x))) =p_v(\DKN-f_v(x)).
	\end{split}
 	\end{equation*}
Since $p_v$ attains its maximum uniquely at $x_v$, it must hold that $\DKN-f_v(x_v)=x_v$. This, combined with the fact that $f(x_v)>x_v$, proves \cref{eq:x_v}.	 

\proofstep{Step 2:} We claim that the intervals in \cref{eq:isoper_interv} are also minimal among finite unions of closed intervals. Let now 
	\begin{equation*}
	E=\bigcup_{i=1}^n [a_i, b_i],\quad n\geq 2,
	\end{equation*}
with $a_1\geq 0$, $b_n\leq \DKN$, and $b_{i-1}<a_i<b_i<a_{i+1}$. Denote by $v_i$ the measure $\mKN([a_i,b_i])$. 

We will move each interval to the \emph{left} or to the \emph{right}, keeping the measure constant and lowering the perimeter. Notice that at least one of the following conditions holds true:
	\begin{equation*}
	a_1< x_{v_1} \qquad \text{or} \qquad a_n > x_{v_n};
	\end{equation*}
indeed, if $a_1\geq x_{v_1}$, then $a_n> b_1\geq f(x_{v_1})>\frac{\DKN}{2}> x_{v_n}$ (here \cref{eq:x_v} has been used). Up to a reflection, we can assume without loss of generality that $a_1 < x_{v_1}$. Then we define $E_0$ as
	\begin{equation*}
	E_0\doteq [0,f_{v_1}(0)]\cup \bigcup_{i=2}^n [a_i,b_i].
	\end{equation*}
$E_0$ now has the same measure as $E$ and smaller perimeter. If $n=2$, we skip to the end of the procedure; if otherwise $n>2$, we proceed inductively in the following way: at each step $1\leq j\leq n-2$, the set $E_{j-1}$  will be the union of $n+1-j$ closed intervals:
	\begin{equation*}
	E_{j-1}=\bigcup_{i=1}^{n+1-j} [a^{j}_i, b^{j}_i],\qquad v_i^{j}\doteq \mKN \pths*{[a^{j}_i, b^{j}_i]},
	\end{equation*}
with $a^{j}_1=0$. We consider the second of those intervals: 
	\begin{itemize}
		\item if $a^{j}_2\leq x_{v_2^{j}}$, then we replace $[0, b^{j}_1]$ and $[a^{j}_2, b^{j}_2]$ with $[0, f_{v_1^j+v_2^j}(0)]$.
		\item if $a^{j}_2 > x_{v_2^{j}}$, then we replace $[a^{j}_2, b^{j}_2]$ and $[a^{j}_3, b^{j}_3]$ with $\bkts*{f^{-1}_{v_2^j+v_3^j}( b^{j}_3), b^{j}_3}$.
	\end{itemize}
The new set $E_j$ is a union of $n-j$ closed intervals, having the same $\mKN$-measure of $E_{j-1}$ and smaller or equal perimeter.

At the end of the procedure, we are left with the union of two intervals; applying the same argument once again, the final set $\tilde{E}$ is either the interval $[0, f_v(0)]$ (in which case the claim is proven), or a union of type $[0, \tilde{b}]\cup [\tilde{a},\DKN]$. In the latter case, however, we can repeat the above argument for the interval $[\tilde{b},\tilde{a}]$ and the measure $1-v$: we move it to the left or to the right applying the same criterion as before, and take the complementary in $\JKN$. This is an interval of the same type as \cref{eq:isoper_interv}, with the same measure of $E$ but lower perimeter.

\proofstep{Step 3:} Finally, we show that the intervals in \cref{eq:isoper_interv} are also minimal among countable unions of disjoint intervals. Assume $E=\bigcup_{i\in\N}[a_i,b_i]$. Since $E$ has finite perimeter, the only accumulation points for the $a_i$'s can be $0$ and $\DKN$. Assume $0$ is an accumulation point; fix $\bar{\imath}$ such that $b_{\bar{\imath}}<\frac{\DKN}{4}$ and let $I\doteq \brcs*{i\in\N\stset b_i \leq b_{\bar{\imath}}}$. Let $\bar{E}\doteq \bigcup_{i\in I}[a_i, b_i]$ and $\bar{v}\doteq \mKN(\bar{E})$. The set 
	\begin{equation*}
	[0,f_{\bar{v}}(0)]\cup \bigcup_{i\in\N\setminus I}[a_i, b_i]
	\end{equation*}
has the same measure and lower perimeter than $E$. Repeating, if necessary, the procedure at $\DKN$, we find a set which is a finite union of closed intervals and lowers the perimeter of $E$, so we can recover the result from \proofstep{Step 2}.
\end{proof}


\subsection{Rearrangements and symmetrizations}
Throughout the section $\Xdm$ will be a metric measure space with $\meas(\Xmms)=1$ and $\Omega\subset \Xmms$ an open subset.

\label{subsec:rearr}

\begin{definition}[Distribution function]
	Let $u:\Omega\to \R$ be a measurable function. We define its \emph{distribution function} $\mu=\mu_u:[0,+\infty) \to [0,\meas(\Omega)]$ as 
	\begin{equation*}
	\mu(t)\doteq \meas (\brcs{\abs{u}>t}).
	\end{equation*}
\end{definition}

\begin{remark}\label{remark:diff_kes}
Our definition of distribution function differs from the one adopted in \cite{Kesavan2006}, but coincides instead with the one used in the original paper by Talenti \cite{Talenti1979}: indeed,  \cite{Kesavan2006} defines $\mu(t)$ as the measure of the superlevel $\brcs{u>t}$ for any $t\in\R$.
\end{remark}

\begin{definition}[Decreasing rearrangement $u^\sharp$]
	Let $u:\Omega\to \R$ be a measurable function. We define $u^\sharp: [0,\meas(\Omega)]\to {[0,\infty]}$ as
	\begin{equation}\label{eq:Defusharp}
		u^\sharp(s)\doteq
		\begin{cases}
		\ess \sup { |u|}		&\text{if $s=0$}\\
		\inf\brcs*{t\in[0,+\infty) \stset \mu_u(t)<s} 	&\text{if $s>0$}
		\end{cases}.
	\end{equation}	 
\end{definition}

The decreasing rearrangement $u^\sharp$ plays the role of a generalized inverse of the distribution function $\mu=\mu_u$: 
\begin{itemize}
\item if $\mu$ is continuous at $\bar{t}$ with $\mu(\bar{t})=\bar{s}$, and $\mu$ is not constant in any interval of the type $[\bar{t},\bar{t}+\delta)$ with $\delta>0$, then $u^\sharp (\bar{s})=\bar{t}$;
\item if $\mu$ is continuous at $\bar{t}$ with $\mu(\bar{t})=\bar{s}$, and $[\bar{t},\bar{t}+\bar{\delta})$ is the largest interval of this type on which $\mu$ is constant, then $u^\sharp(\bar{s})=\bar{t}+\bar{\delta}$;
\item if $\mu$ has a jump discontinuity at $\bar{t}$, with
	$
	\lim_{\tau\to \bar{t}^\pm}\mu(\tau)=\bar{s}^\pm,
	$
then $u^\sharp(s)=\bar{t}$ for any $s\in (\bar{s}^-,\bar{s}^+]$.
\end{itemize}
As the name itself suggests, $u^\sharp$ can be easily shown to be non-increasing; moreover, it is by definition left-continuous.

Finally, we define the \emph{Schwarz symmetrization} of a function $u$ defined on a $\RCD(K,N)$ space. Notice that the condition $\CD(K,N)$ on curvature and dimension, together with the assumption that $\Xdm$ is essentially non-branching, would be enough to ensure a \PSz{} inequality, as shown in \cite{Mondino2019}. 

\begin{definition}[$(K,N)$-Schwarz symmetrization]
Let $\Xdm$ be a \mms{} satisfying the $\RCD(K,N)$ condition for some $K>0$ and $N\in (1,\infty)$. Let $\Omega\subset \Xmms$  be a Borel subset with measure $\meas(\Omega)=v\in [0,1]$ and $u:\Omega\to \R$ be a Borel measurable function. Let $R=R_{K,N,v}>0$ be such that $\mKN([0,R])=v$. We define the $(K,N)$-Schwarz symmetrization $u^\star_{K,N}=u^\star:[0,R]\to[0,\infty]$ as $u^\star\doteq u^\sharp \circ \HKN$; explicitly:
	\begin{equation}\label{eq:defustar}
	u^\star(x)\doteq u^\sharp (\mKN([0,x])).
	\end{equation} 
\end{definition} 

\begin{remark}
Being the composition of $\HKN$, which is increasing, and $u^\sharp$, which is non-increasing, $u^\star$ is still a non-increasing function. 
\end{remark}

We state here a collection of useful facts concerning the decreasing rearrangement of a function: 
these are quite standard and can be found for instance in \cite[Chapter 1]{Kesavan2006} in the context of Euclidean spaces (grounding on a slightly different definition of $\mu_u$, see our \cref{remark:diff_kes}); the proofs contained there still work with very few straightforward modifications.

\begin{proposition}\label{prop:equimeas}
Let $\Xdm$ be an $\RCD(K,N)$ metric measure space for some $K>0$ and $N\in (1,\infty)$,  with $\meas(\Xmms)=1$,  let $\Omega\subset \Xmms$ be a Borel subset  with measure $\meas(\Omega)=v\in [0,1]$ and let $u:\Omega\to \R$ be a measurable function; let $u^\sharp:[0,v]\to { [0,\infty]}$ be its decreasing rearrangement and $u^\star:[0,\HKN^{-1} (v)] \to { [0,\infty]}$ be its $(K,N)$-Schwarz symmetrization. Denote $R=R_{K,N,v}\doteq \HKN^{-1} (v)$.
	\begin{enumerate}[(a)]
	\item\label{item:equimeas_a} $u$, $u^\sharp$ and $u^\star$ are \emph{equimeasurable}, in the sense that
		\begin{equation*}
		\meas\pths*{\brcs{\abs{u}>t}}=\leb^1\pths{\brcs{u^\sharp >t}}=\mKN\pths{\brcs{u^\star >t}}
		\end{equation*}
		for all $t>0$. The same identities hold true with the symbols $\geq$, $<$,  $\leq$ instead of $>$.
	\item\label{item:equimeas_b}  If $u\in L^p(\Omega,\meas)$ for some $1\leq p\leq \infty$, then $u^\sharp \in L^p([0,v],\leb^1)$ and $u^\star \in L^p([0,R_{K,N,v}],\mKN)$. The converse implications also hold. In that case, moreover,
	\begin{equation*}
	\norm*{u}_{L^p(\Omega,\meas)}=\norm{u^\sharp}_{L^p([0,v],\leb^1)}=\norm{u^\star}_{L^p([0,R_{K,N,v}],\mKN)}.
	\end{equation*}
	\item If $u, v \in L^p(\Omega,\meas)$ for some $1\leq p \leq \infty$, then 
		\begin{equation*}
		\norm[\big]{u^\star-v^\star}_{L^p([0,R_{K,N,v}],\mKN)}=\norm[\big]{u^\sharp-v^\sharp}_{L^p([0,v],\leb^1)}\leq \norm*{u-v}_{L^p(\Omega,\meas)}.
		\end{equation*}
	\end{enumerate}
\end{proposition}

\begin{lemma}\label{lemma:KesProp1_2_2}
Let $\Omega\subset \Xmms$ have finite measure; let $f:\Omega\to\R$ be integrable and let $E\subset \Omega$ be measurable. Then:
	\begin{equation*}
	\int_E f \dmeas	\leq \int_0^{\meas(E)} f^\sharp (s) \, \de s.
	\end{equation*} 
Moreover, if $f$ is non-negative, equality holds if and only if $(\restr{f}{E})^\sharp\equiv \restr{(f^\sharp)}{[0,\meas\pths{E}]}$.
\end{lemma}

\begin{proof} 
The proof is analogous to the one proposed in \cite[Chap. 1]{Kesavan2006} in Euclidean setting, we report it briefly for the reader's convenience. Preliminarily, we observe that
	\begin{equation*}
	\int_E f \dmeas	\leq \int_E \abs{f} \dmeas 
	\quad \text{and} \quad
	f^\sharp=\abs{f}^\sharp,
	\end{equation*}
thus we can assume without loss of generality that $f$ is non-negative.
\\First notice that, by equimeasurability,
	\begin{equation}\label{equation:KesProp1_2_2}
	\int_{E} f \dmeas = \int_0^{\meas\pths{E}} (\restr{f}{E})^\sharp (s) \,\de s.
	\end{equation}
Moreover, for any $t\in\R$, we have:
	\begin{equation*}
	\brcs*{x\in E \stset \restr{f}{E}>t} = E\cap \brcs*{x \in \Omega \stset f >t} \subset \brcs*{x \in \Omega \stset f >t},
	\end{equation*}
	thus whenever $s<\meas\pths{E}$:
	\begin{equation*}
	\brcs*{t>0\stset \meas \pths[\big]{ \restr{f}{E}>t}<s} \supset \brcs*{t>0\stset \meas \pths*{ f>t}<s}.
	\end{equation*}
	As a consequence, taking the infimum of the two sets in the previous inclusion, we get the inequality
	\begin{equation*}
	(\restr{f}{E})^\sharp(s)\leq f^\sharp (s),
	\end{equation*}
	which gives, together with \cref{equation:KesProp1_2_2}, the desired result.
\end{proof}

Finally, we give a (necessary and) sufficient condition for a function to coincide with its $(K,N)$-Schwarz symmetrization.

\begin{lemma}\label{lem:decr_rearr}
Let $\phi:\JKN\to [0,+\infty)$ be a {non-increasing} and {non-negative} function. Then $\phi^\star(x)=\phi(x)$ for all $x\in\JKN\setminus L$, where $L$ is a countable set.
\end{lemma}

\begin{proof}
The claim is equivalent to showing that $\phi^\sharp=\phi\circ \HKN^{-1}$ except on a countable set, that is:
	\begin{equation}\label{eq:decr_rearr}
	\inf \brcs*{t\stset \mKN\pths{\brcs{\phi>t}}<s} = \phi \circ \HKN^{-1}(s) , \qquad s\in [0,1]
	\end{equation}
out of a countable set. Let $L\subset \JKN$ be the set of points where $\phi$ is not left continuous (which is countable since $\phi$ is nonincreasing), and fix any $s\in [0,1] \setminus \HKN(L)$.
\\If $\mKN\pths{\brcs{\phi>t}}<s$ for some $t$, then $\mKN\pths{\brcs{\phi>t}}<\mKN([0,\HKN^{-1}(s)])$ and thus 
$$\brcs*{\phi>t} \subsetneq [0,\HKN^{-1}(s)].$$
We infer that $\phi(\HKN^{-1}(s))\leq t$, and thus
	\begin{equation}\label{eq:phiHKNPf1}
	\phi \circ \HKN^{-1}(s) \leq \inf \brcs*{t\stset \mKN\pths{\brcs{\phi>t}}<s}, \quad \forall s\in [0,1] \setminus \HKN(L).
	\end{equation}
 Assume by contradiction	that the inequality in \eqref{eq:phiHKNPf1} is strict for some $s_{0}\in (0,1]\setminus \HKN(L)$. Then there exists $\epsilon>0$ such that
\begin{equation}\label{eq:phiHKNPf2}
	\phi \circ \HKN^{-1}(s_{0})+\epsilon < \inf \brcs*{t\stset \mKN\pths{\brcs{\phi>t}}<s_{0}}.
	\end{equation}
Since by assumption $\phi$ is left-continuous at $\HKN^{-1}(s_{0})$, we can find $\sigma< s_{0}$ such that 
	\begin{equation}\label{eq:phiHKNPf3}
	\phi \pths{ \HKN^{-1}(\sigma)}< \phi\pths{\HKN^{-1}(s_{0})}+\epsilon.
	\end{equation}
Since $\brcs{\phi>\phi\pths{ \HKN^{-1}(\sigma)}}\subset [0,\HKN^{-1}(\sigma)]$, we infer that
	\begin{equation}\label{eq:phiHKNPf4}
	\mKN(\brcs{\phi>\phi\pths{ \HKN^{-1}(\sigma)}})\leq \sigma < s_{0}.
	\end{equation}
The combination of 	\eqref{eq:phiHKNPf2}, \eqref{eq:phiHKNPf3} and \eqref{eq:phiHKNPf4} yields the contradiction
$$
\phi \pths{ \HKN^{-1}(\sigma)}<  \inf \brcs*{t\stset \mKN\pths{\brcs{\phi>t}}<s_{0}} \leq \phi \pths{ \HKN^{-1}(\sigma)}.
$$
This concludes the proof.
\end{proof}


\subsection{Poisson problem on the model space}
\label{subsec:equation_on_model_space}

As already mentioned, the main content of this note is a comparison between the symmetrization of the solution of an elliptic problem on $\Xdm$ and the solution of a symmetrized problem on the model space. We define here the ``model problem'' on the unidimensional space $\JKN$ (see for example \cite[Section 3]{Ambrosio} for more details about Laplacians on weighted spaces).

\begin{notation}[Sobolev space on $\JKN$]\label{notation:sob_jkn}
For a subinterval $I\subset \JKN$, we define
	\begin{equation*}
	\W (I,\deu,\mKN) \doteq \brcs*{v\in L^2 (I,\mKN) \stset v' \in L^2 (I,\mKN)},
	\end{equation*}
where $v'$ is the distributional derivative defined by
	\begin{equation*}
	\int_I v \phi' \,\de \leb^1 = -\int_I v' \phi \,\de  \leb^1 \quad \forall \phi\in \C^\infty_c(I).
	\end{equation*}
We will also endow such space with the norm
	\begin{equation*}
	\norm*{v}_{\W (I,\deu,\mKN)}\doteq \norm*{v}_{L^2 (I,\mKN)} + \norm*{v'}_{L^2 (I,\mKN)}.
	\end{equation*}
\end{notation}

\begin{remark}
The Sobolev space in  \cref{notation:sob_jkn} coincides with the local Sobolev space already defined in \cref{def:loc_sob_spaces}, specializing the latter to  the metric measure space $(I,\deu,\mKN)$.
\end{remark}

\begin{definition}[Laplacian on the model space]\label{definition:weight_lapl}
Let $K>0$ and $N\in (1,\infty)$. We define the \emph{weighted Laplacian} $\Delta_{K,N}:\C^2(\mathring{J}_{K,N}){ \cap \C^1({J}_{K,N})} \to\C^0(\mathring{J}_{K,N})$ on the interval $\mathring{J}_{K,N}$ as:
	\begin{equation}\label{eq:mod_lap}
	\begin{split}
	\Delta_{K,N} \eta 	\doteq \eta''+(\log (\hKN))'\eta' =		\eta''+ \frac{\hKN'}{\hKN} \eta'.
	\end{split}	
	\end{equation}
\end{definition}

Notice that, for any $\eta{  \in \C^2(\mathring{J}_{K,N}) \cap \C^1({J}_{K,N})}$ and any function $\phi {  \in \C^1(\mathring{J}_{K,N}) \cap \C^0({J}_{K,N})}$, {using that $h_{K,N}=0$ on $\partial J_{K,N}$,} one has
	\begin{equation*}
	\int_{\JKN} \eta' \phi' \dmKN = - \int_{\JKN}\pths*{\phi \eta'' \hKN + \phi\eta' \hKN'}\,\de \leb^1= - \int_{\JKN} \phi\lapKN \eta \dmKN,
	\end{equation*}
consistently with \cref{definition:weight_lapl}.

\begin{remark}
The coefficient $\frac{\hKN'}{\hKN}=(\log \hKN)'$ appearing in \cref{eq:mod_lap} can be computed explicitly: indeed, for any $x\in\mathring{J}_{K,N}$ we have:
	\begin{equation*}
	\frac{\hKN'(x)}{\hKN(x)}=\sqrt{K(N-1)}\cot\pths*{\sqrt{\tfrac{K}{N-1}} x}.
	\end{equation*} 
\end{remark}

Accordingly with \cref{definition:weight_lapl}, given an interval $I\subset \JKN$ and $f\in L^2(I,\mKN)$, we say that a function $w$ is a weak solution to $-\Delta_{K,N}w=f$ in $I$ (with appropriate boundary conditions) if it solves
	\begin{equation*}
	-w''- \frac{\hKN'}{\hKN} w' =f     \quad \text{in $I$}
	\end{equation*}
in a distributional sense. In particular, we will be interested in the following Dirichlet problem:

\begin{definition}\label{def:poiss_model}
Let $I\doteq [0,r_1)$ with $0<r_1<\pi\sqrt{\frac{N-1}{K}}$ and let $f\in L^2(I,\mKN)$. We say that $w\in \W (I,\deu,\mKN)$ is a weak solution to
	\begin{equation}\label{eq:poiss_model}
	\left\lbrace
		\begin{aligned}
		-\Delta_{K,N}w&=f	\qquad \text{in $I=[0,r_1)$}\\
		w(r_1)&=0
		\end{aligned}
	\right.
	\end{equation}
if:
	\begin{enumerate}[(i)]
	\item $\displaystyle \int_{[0,r_1]}w'\phi'\,\dmKN = \int_{[0,r_1]}f \phi \,\dmKN$ for any $\phi\in \C^\infty_c([0,r_1))$;
	\item {Boundary condition:} $w\in \W_0 ([0,r_1),\deu,\mKN)$, where the latter space is the closure of ${\C^\infty_c([0,r_1))}$ in the topology of ${\W (I,\deu,\mKN)}$.
	\end{enumerate}
\end{definition} 

\begin{remark}
The intuition behind this choice of boundary conditions is the following. When $N$ is an integer, we think of $(\JKN, \deu, \mKN)$ as the sphere $\mathbb{S}=\mathbb{S}_{K}^{N}$ of dimension $N$ and Ricci curvature $K$. Consider a geodesic ball $B_{r_1}(p)\subset \mathbb{S}$; we look for radial solutions $\hat{w}(x)=w(\dist(x,p))$ of the Dirichlet problem 
	\begin{equation*}
	\begin{cases}
	-\Delta_{\mathbb{S}} \hat{w}(x) = f(\dist(x,p)) & \text{on $B_{r_1}(p)$}\\
	\hat{w}=0		& \text{on $\partial B_{r_1}(p)$}
	\end{cases}.
	\end{equation*}
Then the condition $w(r_1)=0$ comes from the Dirichlet condition on $\partial B_{r_1}(p)$.
\end{remark}
In the next proposition, we give an explicit solution to the problem in \eqref{eq:poiss_model}. 
\begin{proposition}\label{prop:sol_mod_space}
Let $I=[0,r_1)$ with $0 < r_1 < \pi\sqrt{\frac{N-1}{K}}$. Let $f\in L^2(I,\mKN)$. The problem in \cref{eq:poiss_model} admits a unique weak solution $w\in \W (I,\deu,\mKN)$, which can be represented as
	\begin{equation}\label{eq:repr_form_1}
	w(\rho)	= \int_\rho^{r_1} \frac{1}{\hKN(r)}\int_0^{r} f(s)\,\dmKN(s) \,\de r , \quad \forall  \rho\in[0,r_1],
	\end{equation}
or equivalently as
	\begin{equation}\label{eq:repr_form_2}
	w(\rho)= \int_{\HKN(\rho)}^{\HKN(r_1)} \frac{1}{\isoKN^2(\sigma)}\int_{0}^\sigma f\circ\HKN^{-1}(t)\,\de t \,\de \sigma, \quad \forall  \rho\in[0,r_1].
	\end{equation}
\end{proposition}

\begin{proof} As a preliminary fact, notice that the two expressions are actually equivalent, since 
	\begin{equation*}
	\begin{split}
	\int_\rho^{r_1} \frac{1}{\hKN(r)}\int_0^{r} f(s)\,\dmKN(s) \,\de r &=  \int_\rho^{r_1} \frac{1}{\hKN^2(r)}\pths*{\int_0^{r} f(s)\hKN(s)\,\de s} \hKN(r)\,\de r\\
	&= \int_{\HKN(\rho)}^{\HKN(r_1)} \frac{1}{\hKN^2(\HKN^{-1}(\sigma))}\int_{0}^\sigma f\circ\HKN^{-1}(t)\,\de t \,\de \sigma
	\end{split}
	\end{equation*}
and by \cref{lemma:isoprof_model} it holds that $\isoKN=\hKN\circ \HKN^{-1}$. We have used the change of variables $t=\HKN(s)$ in the internal integral and the change of variables $\sigma=\HKN(r)$ in the external integral.

We first show that a weak solution must coincide with the function in \cref{eq:repr_form_1}, and then we prove that such function is actually a solution to \cref{eq:poiss_model}.

\proofstep{Step 1:} Let $w\in \W (I,\deu,\mKN)$ be a weak solution to \cref{eq:poiss_model}. We prove that the weak derivative of $w$ coincides $\mKN$-a.e. with the function
	\begin{equation*}
	g(x)\doteq -\frac{1}{\hKN(x)}\int_0^x f(s) \,\dmKN(s).
	\end{equation*}
Indeed, for any test function $\phi\in \C^\infty_c([0,r_1))$ one has, by the Fubini-Tonelli Theorem:
	\begin{equation}\label{eq:ft_weak}
	\begin{split}
	\int_I (-g(x))\phi'(x) \,\dmKN(x) &= \int_0^{r_1} \pths*{\int_0^{r_1} \chi_{[0,x]}(s)f(s) \frac{\phi'(x)}{\hKN(x)} \,\dmKN(s)}\,\dmKN(x)  \\
	&= \int_0^{r_1} f(s) \pths*{\int_s^{r_1}\phi'(x) \,\de \leb^1(x)} \,\dmKN(s) =\\
	&= - \int_0^{r_1} f(s) \phi(s) \,\dmKN(s).
	\end{split}
	\end{equation}
Thus, since $w$ is a weak solution to \cref{eq:poiss_model}, for any $\phi\in \C^\infty_c([0,r_1))$
	\begin{equation}\label{eq:weak_der}
	\int_I \bkts*{ g(x) - w'(x)} \hKN(x) \phi'(x) \, \de \leb^1(x) = 0.
	\end{equation}
By a classical result (see for example \cite[Lemma 8.1]{Brezis2011}), there exists a constant $C\in \R$ such that $w'(x)\hKN(x)=g(x)\hKN(x)+C$ for $\mKN$-a.e.~$x\in I$. This however implies that for any $\phi\in \C^\infty_c([0,r_1))$
	\begin{equation*}
	0= C\int_0^{r_1} \phi'(x) \, \de \leb^1(x)=C \phi(0),
	\end{equation*}
hence $C=0$.  \\Now $w$ is a $\W (I,\deu,\mKN)$ function, thus in particular it belongs to $W^{1,2}((\epsilon,r_1),\deu,\leb^1)$ for any $\epsilon>0$; moreover, $w$ satisfies $w'=g$ a.e.~and $w(r_1)=0$. Thus, by well known results about Sobolev functions on intervals (see \cite[Theorem 8.2]{Brezis2011}, $w$ coincides with the function in \cref{eq:repr_form_1} for $\mKN$-a.e.~$\rho \in I$.

\proofstep{Step 2:} Let now $w$ be defined as in \cref{eq:repr_form_1}. Since the integrand is continuous on $(0,r_1]$, $w$ is a $\C^1$ function on $(0,r_1]$ (with $w(r_1)=0$). By straightforward computations, we show that $w$ and $w'$ are $L^2(I, \mKN)$ functions. Indeed, by \Holder{} inequality we have that
	\begin{equation}\label{eq:sol_holder}
	\int_0^r \abs{f(s)}\, \dmKN(s)\leq \norm{f}_{L^2(I,\mKN)} \HKN(r)^{\frac{1}{2}};
	\end{equation}
 thus, by \cref{lem:asympt}, 
	\begin{equation*}
	\begin{split}
	\abs{w(\rho)}&\leq \norm{f}_{L^2}\int_\rho^{r_1}\frac{\HKN(r)^{\frac{1}{2}}}{\hKN(r)} \,\de r 	\leq C_1 \norm{f}_{L^2} \int_\rho^{r_1} \frac{r^{\frac{N}{2}}}{r^{N-1}}\,\de r=C_2 \norm{f}_{L^2} \pths*{r_1^{2-\frac{N}{2}}-\rho^{2-\frac{N}{2}}},
	\end{split}	
	\end{equation*}
where $C_1$ and $C_2$ are constants depending only on { $r_{1}\in \left(0, \pi \sqrt{\frac{N-1}{K}} \right)$}, $K>0$ and $N\in (1,\infty)$.
Consequently,
	\begin{equation*}
	\int_0^{r_1} \abs{w}^2\dmKN \leq C_3^2\norm{f}_{L^2}^2 \int_0^{r_1} \pths*{r_1^{4-N}+\rho^{4-N}}\hKN(\rho) \,\de \rho,
	\end{equation*}
which is finite, again by \cref{lem:asympt}.
Moreover, exploiting again \cref{eq:sol_holder}, 
	\begin{equation*}
	\int_0^{r_1} \abs{w'}^2\dmKN \leq  \norm{f}_{L^2}^2 \int_0^{r_1}\frac{\HKN(r)}{\hKN(r)} \,\de r \leq C_4  \norm{f}_{L^2}^2 \int_0^{r_1} \frac{r^N}{r^{N-1}} \,\de r,
	\end{equation*}
which is finite.

By the fact that $\C^{\infty}_c((-\epsilon, r_1))$ is dense in $W^{1,2}_0((-\epsilon, r_1),\deu, \leb^1)$ for any fixed $\epsilon>0$, and noticing that convergence in $W^{1,2}((-\epsilon, r_1),\deu, \leb^1)$ is stronger than convergence in $W^{1,2}_0((-\epsilon, r_1),\deu, \mKN)$, 
we can conclude that the boundary condition in \cref{def:poiss_model} is satisfied.

Finally, by tracing back the identity in \cref{eq:ft_weak}, the very same argument shows that $w$ is a weak solution to \cref{eq:poiss_model}.
\end{proof}

\begin{remark}
Notice that the derivation of the solution still works in the case $r_1=\pi \sqrt{\frac{N-1}{K}}$, provided that the following compatibility condition on $f$ holds true:
	\begin{equation*}
	\int_{\JKN} f \dmKN =0.
	\end{equation*}
However, this case will not be treated in this article.
\end{remark}


\section{A Talenti-type comparison theorem for \texorpdfstring{$\RCD(K,N)$}{RCD(K,N)} spaces}
\label{sec:comparison} 
In this Section, we prove a version of Talenti's comparison theorem in the $\RCD$ setting. The proof is along the lines of (and generalises to the  $\RCD$ setting)  the approach  in the Euclidean framework, e.g. \cite[Section 3.1]{Kesavan2006}. Apart from the technical difficulties of working in a non-smooth setting, the key point is to replace the Euclidean isoperimetric inequality by the Levy-Gromov isoperimetric inequality.

Let $(\Xmms,\dist,\meas)$ be a \mms{} verifying the $\RCD(K,N)$ condition for some $K> 0$ and $N\in (1,\infty)$.
Let $\Omega\subset \Xmms$ be an open domain.

\begin{assumption}\label{assumption:eform}
From now on we will assume $\Eform: L^2(\Xmms, \meas) \times L^2(\Xmms,\meas) \to [-\infty,\infty]$ to be a non-negative definite bilinear form satisfying the following properties:
	\begin{enumerate}[(a)]
		\item Strong locality: $\Eform (u, v) = 0$ whenever $u(x) (v(x)+c)=0$ for $\meas$-a.e.~$x\in\Xmms$, for some constant $c\in\R$.
		\item $\alpha$-uniform ellipticity: there exists $\alpha>0$ such that for any $u\in L^2(\Xmms,\meas)$
			\begin{equation}\label{eq:unif_ellipt}
			\Eform(u,u)\geq \alpha\, \Ch(u,u).
			\end{equation}
		\item $\Eform$ is of order 1: there exists $\beta > 0$ such that $ \Eform (u,u)\leq \beta \norm{u}_{\W \Odm}^2$ for every $u\in \W \Odm$.
	\end{enumerate}
\end{assumption}

\begin{definition}[Domain of $\Eop$]\label{def:domain_of_LE}
Let $\Eform$ be a uniformly elliptic bilinear form as in \cref{assumption:eform}. We define the domain of $\Eop$ as the set
	\begin{equation}\label{eq:domain_of_LE}
	D_\Omega(\Eop)\doteq \brcs*{ u \in W^{1,2}\Odm 
		\stset \text{%
		$\exists f \in L^2(\Omega,\meas)$ such that $\Eform(u,v)=\int_\Omega f v \dmeas$ for all $v\in W^{1,2}_0(\Omega)$%
		}}.
	\end{equation}
If $u\in D_\Omega(\Eop)$ and $f$ satisfies the condition in \cref{eq:domain_of_LE}, we write $-\Eop(u)=f$.
\end{definition}

\begin{definition}[Dirichlet problem on $\Odm$]
 Let $\Eform$ be a uniformly elliptic bilinear form as in \cref{assumption:eform}; let $\Omega \subset \Xmms$ be an open domain and let $f\in L^2(\Omega,\meas)$. We say that a function $u\in W^{1,2}\Xdm$ is a weak solution to the Dirichlet problem
	\begin{equation*}
	\begin{cases}
		-\Eop(u)=f 	& \text{in $\Omega$}\\
		u=0						& \text{on $\partial \Omega$}
	\end{cases}
	\end{equation*} 
if $u\in W^{1,2}_0(\Omega)$ and
	\begin{equation}\label{eq:def_weak_sol}
	\Eform(u,v)=\int_{\Omega} f v \dmeas, 	\qquad \text{for any $v\in W^{1,2}_0 (\Omega)$.}
	\end{equation}
\end{definition}

\begin{remark}
An alternative (but slightly less general) approach would be to adopt the language of differential calculus on metric measure spaces, as introduced for example in \cite{Gigli2018}.
In particular, let $A$ be an element of the $L^2(\Xmms)$-normed $L^\infty(\Xmms)$-module $L^2(T^\ast \Xmms)\otimes L^2(T^\ast \Xmms)$ and assume it is concentrated on $\Omega$. Assume there exists $\alpha>0$ such that for any $X \in \restr{L^2(T\Xmms)}{\Omega}$
		\begin{equation*}
			A(X,X)\doteq A(X\otimes X) \geq \alpha \abs{X}^2,
		\end{equation*}
		where we have denoted by $\abs{\cdot}$ the pointwise norm of $X$, and by $\restr{L^2(T\Xmms)}{\Omega}$ the sub-module of the tangent module whose elements are concentrated on $\Omega$. Recall now that for an infinitesimally Hilbertian \mms{} $\Xmms$ and a function $u\in W^{1,2}\Xdm$ we can define the gradient $\nabla u \in L^2(T\Xmms)$ as the image of the differential $\de u\in L^2(T^\ast \Xmms)$ through the canonical isomorphism between the two $L^\infty$-modules.
If we denote by $\Eform_A: W^{1,2}_0(\Omega)\times W^{1,2}_0(\Omega)\to \R$ the bilinear form defined by 
	\begin{equation*}
	\Eform_A(u,v)\doteq \int_{\Omega} A(\nabla u, \nabla v) \dmeas,
	\end{equation*}
then for any $f\in L^2(\Omega, \meas)$, we say that $u$ is a weak solution to the equation $-\mcL_{\Eform_A}(u)=f$ if 
	\begin{equation*}
	\Eform_A(u,v)=\int_{\Omega} A(\nabla u, \nabla v) \dmeas=\int_\Omega f v\dmeas, \quad \forall v\in W^{1,2}_0(\Omega).
	\end{equation*}
\end{remark}

 
Before passing to the proof of the main comparison theorem, we establish few auxiliary results. We begin with a simple \namecref{lemma:KesLem2_2_1} which only requires $\pths{\Xmms,\meas}$ to be a measure space and $\Omega \subset \Xmms$ to be measurable {with finite measure}.
\begin{lemma}\label{lemma:KesLem2_2_1}
Let $f,u\in L^2(\Omega, \meas)$, with $\Omega \subset \Xmms$ measurable domain {with finite measure}. Define
	\begin{equation*}
	F(t)\doteq \int_{\brcs{u>t}} (u-t) \, f\,  \dmeas , \quad \forall t\in\R.
	\end{equation*}
Then $F$ is differentiable out of a countable set $C\subset\R$, and
	\begin{equation*}
	F'(t)=-\int_{\brcs{u>t}} f \dmeas, \quad \forall t\in\R\setminus C.
	\end{equation*}
\end{lemma}

\begin{proof}
The proof is quite standard, however we recall it for the reader's convenience.
\\ First of all notice that  $\meas(\brcs{u=t})>0$ for an at most countable set $C\in \R$.  
Let $t\in\R\setminus C$ and $h>0$. Then
	\begin{equation*}
	\begin{split}
	F(t+h)-F(t)	&= 	
					\begin{multlined}[t][]
						\int_{\brcs{u>t+h}} \!(u-t)\, f\,	\dmeas   -h	\int_{\brcs{u>t+h}} \!f	\,		\dmeas  \\
					-\bkts*{\int_{\brcs{u>t+h}} 	\!(u-t)\, f	\,\dmeas+
							\int_{\brcs{t<u\leq t+h}} 	\!(u-t)\, f	\,\dmeas}
					\end{multlined}
					\\
				&=-h\int_{\brcs{u>t+h}} \!f	\dmeas-\int_{\brcs{t<u\leq t+h}} \! (u-t)\, f\,	\dmeas,
	\end{split}
	\end{equation*}
which implies
	\begin{equation*}
	\abs*{\frac{F(t+h)-F(t)}{h}+\int_{\brcs{u>t+h}} \!f	\dmeas}\leq \int_{\brcs{t<u\leq t+h}} \!\abs{f}\dmeas.
	\end{equation*}
	The right hand side converges to $0$ by \Holder{} inequality and continuity of the measure, {recalling that $\meas(\brcs{u=t})=0$}. 
	An analogous procedure works for $F(t-h)-F(t)$: we find
	\begin{equation*}
	\abs*{\frac{F(t-h)-F(t)}{-h}+\int_{\brcs{u>t-h}} \!f	\dmeas}\leq \int_{\brcs{t-h<u\leq t}} \!\abs{f}\dmeas;
	\end{equation*}	
	taking the limit as $h\to 0$, this gives the claimed identity.
	\end{proof}

\begin{lemma}\label{lemma:KesThm3_1_1-Step2RCD}
Let $\Omega\subset \Xmms$ be an open domain with finite measure, $\Eform$ be as in \cref{assumption:eform} and $f\in L^2(\Omega,\meas)$. Let $u\in W^{1,2}_0(\Omega)$ be a weak solution to $-\Eop(u)=f$. Then for $\leb^1$-a.e. $t>0$ it holds:
	\begin{equation}\label{equation:KesThm3_1_1-Step2RCD}
	\pths*{-\frac{d}{dt}\int_{\brcs{\abs{u}>t}} \mwug{u} \dmeas}^2 \leq -\frac{1}{\alpha}\mu'(t)\int_{\brcs{\abs{u}>t}} \abs{f} \dmeas,
	\end{equation}
where $\mu=\mu_u$ is the distribution function of $u$ and $\mwug{u}$ denotes the minimal $2$-weak upper gradient of $u$.
\end{lemma}

\begin{proof}
Let $t>0$ be fixed, and consider the following test function:
\begin{equation}\label{eq:test_func}
v_t\doteq (u-t)^+ - (u+t)^- = 
	\begin{cases}
	u-t=\abs{u}-t		&\text{if $u>t$}\\
	0					&\text{if $\abs{u}\leq t$}\\
	u+t=-(\abs{u}-t)	&\text{if $u<-t$}\\
	\end{cases}.
\end{equation}
It is easy to see that $v_t$ still belongs to the space $W^{1,2}_0(\Omega)$, thus it can be used as a test function in \cref{eq:def_weak_sol} to obtain
	\begin{equation*}
	\Eform (u,v_t) = \int_\Omega f v_t \dmeas = \int_{\brcs{u>t}} (u-t) \,f \,  \dmeas  -  \int_{\brcs{{ - u> t}}}  (-u-t) \, f\, \dmeas.
	\end{equation*}
By applying \cref{lemma:KesLem2_2_1} we obtain that, for $\leb^1$-a.e.~$t>0$, $t\mapsto \Eform (u,v_t)$ is differentiable with
	\begin{equation}\label{equation:Uprime}
	-\frac{d}{dt}\Eform (u,v_t)=\int_{\brcs{u>t}} f \dmeas  -  \int_{\brcs{{ u< -t}}} f \dmeas  \leq \int_{\brcs{\abs{u}>t}}  \abs{f} \dmeas.
	\end{equation}
For fixed  $t>0$ and $h>0$,  by bilinearity of $\Eform$ it holds that
	\begin{equation}\label{eq:KesThm3_1_1-1}
	\Eform(u,v_{t+h})-\Eform(u,v_{t})=\Eform(u,v_{t+h}-v_t).
	\end{equation}
Moreover, we can explicitly write
	\begin{equation}\label{eq:KesThm3_1_1-2}
	v_{t+h}-v_t= - \operatorname{sgn}(u) \bkts*{(\abs{u}-t) \chi_{\brcs{t< \abs{u} \leq t+h}} + h\chi_{\brcs{\abs{u}>t+h}}}=
	\begin{cases}
		h 				& \text{if $u<-t-h$}\\
		-(u+t) 			& \text{if $-t-h\leq u < -t$}\\
		0				& \text{if $\abs{u}\leq t$}\\
		-(u-t) 			& \text{if $t<u \leq t+h$}\\
		-h 				& \text{if $u>t+h$}
	\end{cases}.
	\end{equation}
Notice that, by strong locality and bilinearity of $\Eform$, for any $B\in \borel(\Xmms)$
	\begin{equation}\label{eq:KesThm3_1_1-3}
	0=\Eform(u \chi_B, \chi_B) + \Eform(u \chi_{\Xmms\setminus B} , \chi_B)= \Eform (u, \chi_B).
	\end{equation}
In particular, it follows from \cref{eq:KesThm3_1_1-1,eq:KesThm3_1_1-2,eq:KesThm3_1_1-3} that 
	\begin{equation*}
	\begin{split}
	\frac{\Eform(u,v_{t+h})-\Eform(u,v_{t})}{h}&=
	-\frac{1}{h}\bkts*{	\Eform \pths*{u, (u+t) \chi_{\brcs{-t-h\leq u < -t}}}+
						\Eform \pths*{u, (u-t) \chi_{\brcs{t<u \leq t+h}}}}\\
	&=-\frac{1}{h}\bkts*{	\Eform \pths*{u \chi_{\brcs{-t-h\leq u < -t}},u \chi_{\brcs{-t-h\leq u < -t}}}+
							\Eform \pths*{u \chi_{\brcs{t<u \leq t+h}},u  { \chi_{\brcs{t<u \leq t+h}} } } }	.			
	\end{split}
	\end{equation*}
By $\alpha$-uniform ellipticity, then, the following estimate holds true:
	\begin{equation}\label{eq:KesThm3_1_1-4}
	-\frac{1}{\alpha}\frac{\Eform(u,v_{t+h})-\Eform(u,v_{t})}{h} \geq \frac{1}{h} \int_{\brcs{t< \abs{u}\leq t+h}} \mwug{u}^2 \dmeas.
	\end{equation}
Consequently, the following chain of inequalities holds for all $t\in\R$ and $h>0$:
	\begin{equation}\label{eq:preGradBound}
		\begin{split}
		\bigg(\frac{1}{h} \int_{\brcs{t< \abs{u} \leq t+h}} \mwug{u} \dmeas \bigg)^2 & \leq  	\pths*{\frac{1}{h} \int_{\brcs{t< \abs{u} \leq t+h}} \mwug{u}^2 \dmeas}
				\pths*{\frac{\meas\pths*{{\brcs{t< \abs{u} \leq t+h}}}}{h}} \\
		&\leq	\frac{1}{\alpha}\pths*{-\frac{\Eform(u,v_{t+h})-\Eform(u,v_{t})}{h}}
				\pths*{-\frac{\mu(t+h)-\mu(t)}{h}} 	.
		\end{split}
	\end{equation}	
Hence, if $t$ is a differentiability point for $t\mapsto \Eform(u, v_t)$, letting $h\to 0$ and using \cref{equation:Uprime} we get exactly the desired result.
\end{proof}

We now state a suitable version of the coarea formula which can be found in \cite[Remark 4.3]{Miranda2003} for a general version in metric measure spaces, and in \cite[Theorem 2.12]{Mondino2019} for a contextualization in $\RCD$ spaces (the identity in the form we state follows from the latter by a monotone convergence argument).

\begin{proposition}[Coarea formula]
\label{prop:coarea_formula}
Let $\Xdm$ be an $\RCD(K,N)$ space for some $K\in\R$, $N\in (1,\infty)$. Let $\Omega \subset \Xmms$ be an open domain and $u:\Omega \to \R$ be a non-negative function in $W^{1,2}_0(\Omega)$. Then for any $t > 0$
	\begin{equation}\label{equation:coarea_formula}
	\int_{\brcs{u>t}} \mwug{u} \dmeas =\int_t^{\infty} \Per(\brcs{u>r})\, \de r.
	\end{equation}
More generally,  for any Borel function $f:\Omega \to \R$ and for any  $t>0$, it holds that
	\begin{equation*}
	\int_{\brcs{u>t}} f \mwug{u} \dmeas = \int_t^\infty \pths*{ \int f \,\de \Per(\brcs{u>r})} \de r.
	\end{equation*}
\end{proposition}

Notice that the most general theorem works for functions of bounded variation (see again \cite{Miranda2003}). \cref{prop:coarea_formula} follows from such a ${\rm BV}$ version combined with \cite[Remark 3.5]{Gigli2016}, and by the fact that the $\CD(K,N)$ condition with $N\in(1,\infty)$ implies properness of the   space (implies local doubling, thus properness \cite{Sturm2006b}).

Next, we recall the Lévy-Gromov isoperimetric inequality in $\RCD$ spaces, as obtained by Cavalletti and Mondino in \cite{Cavalletti2017a} (for the Minkowski content) and in \cite{Cavalletti2018a} (for the perimeter).

\begin{proposition}[Lévy-Gromov inequality]
\label{prop:levy_gromov}
Let $\Xdm$ be an $\RCD(K,N)$ \mms{} with $K>0$ and $N\in (1,\infty)$. Then for any $E\in\msB(\Xmms)$
	\begin{equation}\label{equation:levy_gromov}
	\Per(E)\geq \isoKN(\meas(E)).
	\end{equation}
In particular, the isoperimetric profile of $\Xdm$ is bounded from below by $\isoKN$.
\end{proposition}

By differentiating the coarea formula \eqref{equation:coarea_formula} and exploiting the Lévy-Gromov inequality \eqref{equation:levy_gromov} we get:  

\begin{corollary}
\label{corollary:KesCor2_2_3}
Let $\Xdm$ be an $\RCD(K,N)$ space for some $K>0$, $N\in (1,\infty)$. Let $\Omega \subset \Xmms$ be an open domain and $u:\Omega \to \R$ be a function in $W^{1,2}_0(\Omega)$. Then the map 
	\begin{equation*}
	t \mapsto \int_{\brcs{\abs{u}>t}} \mwug{u} \dmeas
	\end{equation*}
is absolutely continuous and
	\begin{equation*}
	-\frac{d}{dt}\pths*{\int_{\brcs{\abs{u}>t}} \mwug{u} \dmeas}\geq \isoKN\pths*{\meas\pths*{\brcs*{\abs{u}>t}}}= \isoKN\pths*{\mu(t)}.
	\end{equation*}
\end{corollary}

We have now the tools needed to prove our first main result.

\begin{theorem}[A Talenti-type comparison for $\RCD(K,N)$ spaces]\label{theorem:talenti}
Let $\Xdm$ be an $\RCD(K,N)$ space for some $K>0$, $N\in (1,\infty)$,  with $\meas(\Xmms)=1$, and let  $\Omega \subset \Xmms$ be an open domain with measure $\meas(\Omega)=v\in (0,1)$. Let $f\in L^2(\Omega,\meas)$.   Let $\Eform$ be a $\alpha$-uniformly elliptic bilinear form as in \cref{assumption:eform} and assume that  $u\in W^{1,2}_{0}(\Omega)$ is a weak solution to the equation $-\Eop (u)=f$. Let also $w\in W^{1,2}(I,\deu, \mKN)$ be a weak solution (as in \cref{def:poiss_model}) to the problem
	\begin{equation}\label{eq:symm_problem}
	\left\lbrace
	\begin{aligned}
	-\alpha \lapKN w &= f^\star \quad \text{in $I$}\\
	w(r_1)&=0  
	\end{aligned}
	\right.,
	\end{equation}
where $I=[0,r_{v})$, $r_{v}>0$ is such that $\mKN([0,r_{v}))=\meas (\Omega)$, and $f^\star$ is the Schwarz symmetrization of $f$. Then
\begin{enumerate}
\item  $u^\star(x)\leq w(x)$, {for every $x\in [0, r_{v}]$}.
\item For any $1\leq q \leq 2$, the following $L^{q}$-gradient estimate holds: 
	\begin{equation}\label{eq:GradComp}
	\int_\Omega \mwug{u}^q \dmeas	\leq \int_0^{r_{v}}\abs{w'(\rho)}^q \dmKN(\rho).
	\end{equation}
\end{enumerate}
\end{theorem}

\begin{remark}
The Dirichlet problem in \cref{eq:symm_problem} can be explicitly rewritten as
	\begin{equation*}
	\left\lbrace
	\begin{aligned}
	&-w''- \frac{\hKN'}{\hKN} w' =\frac{1}{\alpha}f^\star \quad \text{in $I$}\\
	&w(\HKN^{-1}(\meas(\Omega)))=0  
	\end{aligned}
	\right.,
	\end{equation*}
by the definition of $\lapKN$ and $\HKN$.
\end{remark}

\begin{proof}
{\textbf{Proof of 1.}} By combining \cref{lemma:KesProp1_2_2}, \cref{lemma:KesThm3_1_1-Step2RCD} and \cref{corollary:KesCor2_2_3}, we obtain the following chain of inequalities:
	\begin{equation} \label{eq:main_thm0}
	\begin{split}
	\isoKN(\mu(t))^2&  \leq 
		\pths*{-\frac{d}{dt}{\int_{\brcs{\abs{u}>t}}\mwug{u}\dmeas}}^2   \leq
		- \frac{1}{\alpha} \mu'(t) \int_{\brcs{\abs{u}>t}}  \abs{f} \dmeas  \\
	& \leq - \frac{1}{\alpha} \mu'(t) \int_0^{\mu(t)} f^\sharp (s)\,\de s
	\end{split}
	\end{equation}
for almost every $t>0$, which can be rewritten as
	\begin{equation} \label{eq:main_thm1}
	1\leq- \frac{\mu'(t)}{\alpha \,\isoKN(\mu(t))^2} \int_0^{\mu(t)} f^\sharp (s)\,\de s
	\end{equation}
for almost every $t\in\pths{0, M}$, where $M=\ess \sup u$. For $\xi>0$ let
	\begin{equation}\label{eq:bigF}
	F(\xi)\doteq \int_0^\xi f^\sharp (s) \,\de s.
	\end{equation}
Let now $0\leq \tau'<\tau \leq M$. Integrating \cref{eq:main_thm1} from $\tau'$ to $\tau$ we get
	\begin{equation*}
	\tau-\tau'\leq  \frac{1}{\alpha} \int_{\tau'}^\tau	\frac{F(\mu(t))}{\isoKN(\mu(t))^2}(-\mu'(t)) \,\de t, \qquad 0\leq \tau'<\tau \leq M.
	\end{equation*}	
Using the change of variables $\xi=\mu(t)$ on the intervals where $\mu$ is absolutely continuous, and observing that the integrand is non negative, we obtain
	\begin{equation*}
	\tau-\tau'\leq \frac{1}{\alpha}\int_{\mu(\tau)}^{\mu(\tau')} \frac{F(\xi)}{\isoKN(\xi)^2} \,\de \xi, \qquad 0\leq \tau'<\tau \leq M.
	\end{equation*}
Let us fix $s\in \pths{0,\mu(0)}$ and let $\eta>0$ be a small enough parameter (that will eventually tend to $0$);  
 consider $\tau'=0$ and $\tau=u^\sharp(s)-\eta$.
Notice that, since $u^\sharp(s)$ is the infimum of the $\tilde{\tau}$ such that $\mu(\tilde{\tau})<s$, we have that $\mu(\tau)\geq s$. Using again the non-negativity of the integrand, for any $\eta>0$ we obtain that
	\begin{equation*}
	u^\sharp (s)-\eta\leq \frac{1}{\alpha}\int_s^{\mu(0)} \frac{F(\xi)}{\isoKN(\xi)^2}\,\de \xi, \qquad \forall s\in(0,\mu(0)).
	\end{equation*}
Letting $\eta\downarrow 0$ and enlarging the integration interval, we get:
	\begin{equation}\label{eq:main_thm2}
	u^\sharp (s) \leq \frac{1}{\alpha}\int_s^{\meas(\Omega)} \frac{1}{\isoKN(\xi)^2} \int_0^\xi f^\sharp (t) \,\de t\,\de \xi,
	\qquad \forall s\in (0,\meas(\Omega)).
	\end{equation}
Notice that on $(\mu(0),\meas(\Omega))$ the function $u^\sharp$ vanishes. Finally, by the definition of the symmetrized function $u^\star=u^\sharp \circ \HKN$, we obtain
	\begin{equation*}
	u^\star (x)\leq \frac{1}{\alpha}\int_{\HKN(x)}^{\meas(\Omega)} \frac{1}{\isoKN(\xi)^2} \int_0^\xi f^\star (\HKN^{-1}(t)) \,\de t\,\de \xi,
	\qquad \forall x\in\JKN.
	\end{equation*}
Now we can recognize that the right hand side coincides with the characterization of $w$ we obtained in \cref{eq:repr_form_2} (\cref{subsec:equation_on_model_space}), since $r_{v}$ was chosen so that $\HKN(r_{v})=\meas(\Omega)$. Note that, since the integrand is non-negative, $w$ is non-increasing, and takes the value zero at $r_{v}$.
\medskip

{\textbf{Proof of 2.}} We start by noticing that 
	\begin{equation*}
	\int_\Omega \mwug{u}^q \dmeas	= \int_{\brcs{\abs{u}>0}} \mwug{u}^q \dmeas	,
	\end{equation*}
	since $\mwug{u} = 0$ $\meas$-a.e.~on $\brcs{u=\kappa}$ for any $\kappa\in\R$.
	Let $M:=\ess \sup_{\Omega} |u|$;  fix $t>0$ and  $0<h<M-t$.
By using the \Holder{} inequality (with exponents $\frac{2}{q}$ and $\frac{2}{2-q}$) one gets
	\begin{equation}\label{eq:gradients_proof_1}
	\frac{1}{h} \int_{\brcs{t<\abs{u}\leq t+h}} \mwug{u}^q \dmeas 
	\leq 
	\pths*{\frac{1}{h} \int_{\brcs{t<\abs{u}\leq t+h}} \mwug{u}^2 \dmeas}^{\frac{q}{2}}\pths*{\frac{\meas\pths*{\brcs{t<\abs{u}\leq t+h}}}{h}}^{\frac{2-q}{2}} . 
	\end{equation}
By the very same computations we already performed in \cref{lemma:KesThm3_1_1-Step2RCD}, exploiting the test functions $v_t\in W^{1,2}_0(\Omega)$ defined in \cref{eq:test_func} (see \cref{equation:Uprime,eq:KesThm3_1_1-4,eq:preGradBound}), we can let $h$ tend to zero in \cref{eq:gradients_proof_1} and obtain 
that the map 
	\begin{equation*}
	t\mapsto \int_{\brcs{\abs{u}>t}} \mwug{u}^q \dmeas
	\end{equation*}
is absolutely continuous on $(0,M)$ and thus
\begin{equation}\label{eq:gradients_proof_0}
	\int_\Omega \mwug{u}^q \dmeas = \int_0^M -\frac{d}{dt} \int_{\brcs{\abs{u}>t}}\mwug{u}^q \dmeas \, \de t;
	\end{equation}
	moreover  
	\begin{equation*}
	-\frac{d}{dt} \int_{\brcs{\abs{u}>t}} \mwug{u}^q \dmeas \leq \pths*{\frac{1}{\alpha} \int_{\brcs*{\abs{u}>t}}f\,\dmeas}^{\frac{q}{2}}(-\mu'(t))^{\frac{2-q}{2}}.
	\end{equation*}
Let us now adopt again the notation 	
	\begin{equation*}
	F(\xi)\doteq  \int_0^\xi f^\sharp (s) \,\de s,
	\end{equation*}
as in \cref{eq:bigF}. Exploiting again \cref{lemma:KesProp1_2_2}, we get:
	\begin{equation}\label{eq:gradients_proof_2}
	-\frac{d}{dt} \int_{\brcs{\abs{u}>t}} \mwug{u}^q \dmeas \leq \pths*{\frac{F(\mu(t))}{\alpha}}^{\frac{q}{2}}(-\mu'(t))^{\frac{2-q}{2}}
	\end{equation}
for almost every $t$. In order to obtain a clean term $\mu'(t)$ at the right hand side, we multiply both sides of \cref{eq:gradients_proof_2} with the respective sides of \cref{eq:main_thm1} raised at the power $\frac{q}{2}$. This gives, for almost every $t\in(0,M)$:
	\begin{equation*}
	-\frac{d}{dt} \int_{\brcs{\abs{u}>t}} \mwug{u}^q \dmeas \leq \pths*{\frac{F(\mu(t))}{\alpha \isoKN(\mu(t))}}^{q}(-\mu'(t)).
	\end{equation*}
Inserting this last inequality in \cref{eq:gradients_proof_0} and changing the variables as usual with $\xi=\mu(t)$, the following estimate holds:
	\begin{equation}\label{eq:gradients_proof_3}
	\int_\Omega \mwug{u}^q \dmeas \leq \int_0^{\meas(\Omega)} \pths*{\frac{F(\xi)}{\alpha \isoKN(\xi)}}^q \,\de \xi.
	\end{equation}
	
Finally, we recall that $w$ has an explicit expression we can differentiate: by differentiating \cref{eq:repr_form_1} (with datum $\frac{f^\star}{\alpha}$), we find for all $\rho\in(0,r_{v})$
	\begin{equation*}
	w'(\rho)=-\frac{1}{\hKN(\rho)}\int_0^{\HKN(\rho)}\frac{1}{\alpha} f^\star(\HKN^{-1}(t))\,\de t = -\frac{F(\HKN(\rho))}{\alpha \hKN(\rho)}.
	\end{equation*}
Thus, the following identity holds true:
	\begin{equation}\label{eq:gradients_proof_4}
	\begin{split}
	\int_0^{r_{v}} \abs{w'(\rho)}^q \dmKN &= \int_0^{r_{v}} \pths*{\frac{F(\HKN(\rho))}{\alpha \hKN(\rho)}}^q \hKN(\rho)\,\de\rho =	
	\int_0^{\meas (\Omega)} \pths*{\frac{F(\xi)}{\alpha \isoKN(\xi)}}^q\,\de\xi,
	\end{split}
	\end{equation}
	where we have used the change of variables $\xi=\HKN(\rho)$ and the fact that $\isoKN(\xi)=\hKN(\HKN^{-1}(\xi))$.
	Comparing with \cref{eq:gradients_proof_3}, we obtain the claimed $L^{q}$-gradient estimate.

\end{proof}


\section{Rigidity and Stability}
\label{sec:rig_arig}


\subsection{Rigidity in the Talenti-type theorem}
\label{subsec:rigidity}

Let $u\in W^{1,2}_0(\Omega)$ and $w\in W^{1,2}([0,r_{v}),\deu,\mKN)$ be as in \cref{theorem:talenti}. The next problem we want to approach is the equality case, that is, what we can say about the original \mms{} when $u^\star=w$; in fact, we will prove that if the equality is attained at least at one point, then the \mms{} is forced to have a particular structure, namely it is a spherical suspension. We recall that, in the Euclidean case $\Omega\subset \R^n$, the condition $u^\star=w$ forces $\Omega$ to be a ball and both $u$ and $f$ to be radial.

In order to tackle this question, we recall the definition of a spherical suspension, we state the Rigidity Theorem for the Lévy-Gromov inequality (as proved in \cite{Cavalletti2018a}) and the \PSz{} Theorem for $\RCD(K,N)$ spaces, which was proved in \cite{Mondino2019}.

\begin{definition}[Spherical suspensions]
Let $(B,\dist_B,\meas_B)$ and $(F,\dist_F,\meas_F)$ be geodesic \mms{}s and $f:B\to [0,\infty)$ be a Lipschitz function. Let $\dist$ be the pseudo-distance on $B\times F$ defined by
	\begin{equation*}
	\dist((p,x),(q,y))\doteq \inf\brcs*{L(\gamma) \stset \gamma(0)=(p,x), \gamma(1)=(q,y)},
	\end{equation*}
where, for any absolutely continuous curve $\gamma = (\gamma_B, \gamma_F):[0,1]\to B\times F$,
	\begin{equation*}
	L(\gamma)\doteq \int_0^1 \pths*{\abs{\gamma_B'}^2+(f\circ \gamma_B)^2 \abs{\gamma_F'}^2}^{\frac{1}{2}} \,\de t.
	\end{equation*}
Given $N\geq 1$, we define $B \times_f^N F$ to be the \mms{} 
	\begin{equation*}
	\pths*{(B\times F)\slash\!\sim, \dist, \meas},
	\end{equation*}
where $\sim$ is the equivalence relation associated to the pseudo-distance $\dist$ and $\meas\doteq f^N \meas_B\otimes \meas_F$.

We say that an $\RCD(N-1,N)$ space $\Xdm$ is a \emph{spherical suspension} if it is isomorphic to $[0,\pi]\times^{N-1}_{\sin} \Ymms$ for an $\RCD(N-2,N-1)$ space $\Ydm$ with $\measY(\Ymms)=1$.
\end{definition}

Just for simplicity, the following results are stated in the case of $\RCD(N-1,N)$ spaces;  indeed  when $K>0$ it is not restrictive to assume $K=N-1$ by \eqref{eq:Scaling}.
Notice, moreover, that this assumption only affects the Rigidity statements, while the \PSz{} inequality holds in the very same form for general $K>0$.

\begin{theorem}[Rigidity for Lévy-Gromov, \cite{Cavalletti2018a}]
Let $\Xdm$ be an $\RCD(N-1 ,N)$ space for some $N\in[2,+\infty)$, with $\meas(\Xmms)=1$. 
Assume there exists $\bar{v}\in(0,1)$ such that $\mcI_{\Xdm}(\bar{v})=\mcI_{N-1,N}(\bar{v})$. 
Then $\Xdm$ is a spherical suspension: \ie{}, there exists an $\RCD(N-2,N-1)$ space $\Ydm$ with $\measY(\Ymms)=1$ such that
		\begin{equation*}
		\text{$\Xdm$ is isomorphic as a \mms{} to $[0,\pi]\times^{N-1}_{\sin} \Ymms$}.
		\end{equation*}
\end{theorem}

\begin{theorem}[\PSz\,  for $\RCD(N-1, N)$ spaces,  \cite{Mondino2019}]\label{thm:psz}
	Let $\Xdm$ be an $\RCD(N-1 ,N)$ space for some $N\in[2,+\infty)$, with $\meas(\Xmms)=1$. Let $\Omega\subset \Xmms$ be an open subset with measure $\meas(\Omega)=v\in (0,1)$ and let $r_{v}\in (0,\pi)$ such that $ \mNN([0,r_{v}])=v$. Then,	for every  $p\in (1,\infty)$, the following hold:
	\begin{enumerate}[(i)]
	\item \textnormal{\PSz\, comparison}: for any $u\in W^{1,p}_0(\Omega)$, it holds that $u^\star(r_{v})=0$ and
		\begin{equation}\label{eq:psz}
		\int_0^{r_{v}}\abs{\nabla u^\star}^p \dmNN \leq \int_\Omega \abs{\nabla u}^p \dmeas.
		\end{equation}
	\item \textnormal{Rigidity}: if there exists $u\in W^{1,2}_0(\Omega)$ with $u \not\equiv 0$, achieving equality in  \cref{eq:psz}, then $\Xdm$ is a spherical suspension.
	\item \textnormal{Rigidity for Lipschitz functions}: if there exists $u\in W^{1,2}_0(\Omega)\cap \Lip(\Omega)$ with $u\not\equiv 0$ and $\nabla u\neq 0$ $\meas$-a.e.~in $\spt (u)$, achieving equality in  \cref{eq:psz}, then $\Xdm$ is a spherical suspension and $u$ is radial: that is, $u$ is of the form $u=g(\dist(\cdot,x_0))$, with $x_0$ being the tip of a spherical suspension structure of $X$, and $g:[0,\pi]\to \R$ satisfying $\abs*{g}=u^\star$.
	\end{enumerate}
\end{theorem}

The following rigidity result for the Talenti-type comparison theorem will build on top of the rigidity   in the Lévy-Gromov and \PSz{} inequalities.

\newcommand{\bx}{\bar{x}}
\newcommand{\bs}{\bar{s}}
\newcommand{\ust}{u^\star}
\newcommand{\ush}{u^\sharp}
\newcommand{\wsh}{w^\sharp}
\begin{theorem}[Rigidity for Talenti in $\RCD$]
\label{thm:talenti_rigidity}
Let $\Xdm$ be an $\RCD(N-1,N)$ space for some $N\in [2,\infty)$, with $\meas(\Xmms)=1$, and let $\Omega \subset \Xmms$ be an open domain with measure $\meas(\Omega)=v\in(0,1)$. Let $f\in L^2(\Omega,\meas)$, with $f\not\equiv 0$.
Let $\Eform$ be a $\alpha$-uniformly elliptic bilinear form as in \cref{assumption:eform} and assume that  $u\in W^{1,2}_{0}(\Omega)$ is a weak solution to the equation $-\Eop (u)=f$.
Let also $w\in W^{1,2}(I,\deu,\mNN)$ be a solution to the problem
	\begin{equation*}
	\left\lbrace
	\begin{aligned}
	-\alpha \lapNN w&=f^\star \quad \text{in $I$}\\
	w(r_{v})&=0  
	\end{aligned}
	\right.
	\end{equation*}
where $I=[0,r_{v})$, $r_{v}\in (0,\pi)$ is such that $\mNN([0,r_{v}))=v$, and $f^\star$ is the Schwarz symmetrization of $f$. Assume that $u^\star(\bx)=w(\bx)$ for a point $\bx\in[0,r_{v})$.
Then:
	\begin{enumerate}
		\item $\ust=w$ in the whole interval $[\bx, r_v]$;
		\item $\Xdm$ is a spherical suspension, \ie{} there exists an $\RCD(N-2,N-1)$ space $\Ydm$ with $\measY(\Ymms)=1$ such that $\Xdm$ is isomorphic as a \mms{} to $[0,\pi]\times^{N-1}_{\sin} \Ymms$;
		\item if $\bx=0$, $u\in\Lip(\Omega)$ and $\mwug{u}\neq 0$ $\meas$-a.e.~in $\spt (u)$, then $u$ is radial:  that is, $u$ is of the form $u=g(\dist(\cdot,x_0))$, with $x_0$ being the tip of a spherical suspension structure of $X$, and $g:[0,\pi]\to \R$ satisfying $\abs*{g}=u^\star$.
	\end{enumerate}
\end{theorem}

In order to establish Theorem \ref{thm:talenti_rigidity}, we first prove a preliminary lemma which will also be useful in \cref{subsec:alm_rig}: in the same setting of the Talenti-type Theorem, the difference $w-\ust$ is non-increasing.

\newcommand{\sz}{s_0}
\begin{lemma}
\label{lem:monoton_w-u}
Let $\Xdm$, $\Omega$, $f$, $\Eform$, $u$ and $w$ be as in \cref{theorem:talenti}. 
Then the map $x\mapsto w(x)-\ust(x)$ is non-increasing on $[0,r_v]$.
\end{lemma}

\begin{proof}[Proof of \cref{lem:monoton_w-u}]
Since $w-\ust = (\wsh - \ush)\circ \HKN$, with $\HKN$ strictly increasing, it is enough to show that $\wsh - \ush$ is non-increasing in $[0,\meas(\Omega)]$.
\\Recall that the function $\wsh \colon [0,\meas (\Omega)]\to\R$ can be expressed as:
	\begin{equation}\notag
	\wsh (s) = \frac{1}{\alpha} \int_s^{\meas (\Omega)} \frac{F(\xi)}{\isoKN^2(\xi)}\,\de\xi
	\end{equation}
where $F(\xi)\doteq \int_0^\xi f^\sharp (s)\, \de s$ as usual.
As a preliminary observation, notice that this explicit representation gives some useful information on the regularity and behavior of $\wsh$ itself: 
indeed, $\wsh$ is a continuously differentiable function on $(0,\meas(\Omega))$, and it is strictly decreasing in $[0,r_v]$ (since $f\not\equiv 0$).
Moreover, as a consequence of the \PSz{} Theorem, $\ust$ belongs to $\W_0([0,r_v),\deu,\mKN)$ and it is thus locally absolutely continuous in the interior $(0,r_v)$; the same conclusion thus holds for $\ush$. Hence the result is proved if we can show that 
\begin{equation}\label{eq:w-u-derivative}
\text{$(\wsh-\ush)'\leq 0$ almost everywhere in $(0,\meas(\Omega))$.}
\end{equation}

By the continuity of $\ush$ and by the definition of symmetrization, we have that $\ush (\mu(t)) = t$ for all $t\in(0,M)$ (\ie{} $\mu$ is the right inverse of $\ush$), where $M\doteq \sup u$. In particular, $(\ush\circ \mu)'\equiv 1$ in $(0,M)$. On the other hand, $(\wsh \circ \mu)'\geq 1$ a.e.\ in $(0,M)$ by \cref{eq:main_thm1}. Hence,
	\begin{equation}\notag
	\bkts*{(\wsh)'\circ \mu - (\ush)'\circ \mu} \mu' \geq 0 \qquad \text{a.e.\ in $(0,M)$}.
	\end{equation}
Moreover, $\mu'$ is strictly negative a.e.\ in $(0,M)$, again by the fact that $\bkts{(\ush)'\circ \mu} \mu' = 1$ almost everywhere. This shows that in fact
	\begin{equation} \label{eq:comp_deriv}
	\pths*{\wsh}'(\mu (t))\leq \pths*{\ush}'(\mu (t)) 	\qquad \text{for a.e.\ $t\in(0,M)$}.
	\end{equation}
To be more precise, 	\eqref{eq:comp_deriv} holds at all points $t$  such that  $\mu$ is differentiable at $t$ and  $\ush$ is differentiable at   $\mu(t)$.  
\\Let $\sz\in (0,\meas(\Omega))$ be a point where $\ush$ is differentiable. Since $\ush$ is monotone non-increasing we have that either $(\ush)'(\sz)=0$ or $(\ush)'(\sz)<0$.
\\ If $(\ush)'(\sz)=0$, then  $(\wsh-\ush)'(\sz)= (\wsh)'(\sz)<0$ so the inequality \eqref{eq:w-u-derivative} is proved.
\\If instead $(\ush)'(\sz)<0$ then,  by monotonicity,  $\ush(s)>\ush(\sz)$ for any $s<\sz$. In particular $\sz = \mu(\ush(\sz))$. It is also easily seen that in this case $\mu$ is differentiable at $\ush(\sz)$.  By \cref{eq:comp_deriv}, we conclude that $(\wsh-\ush)'(\sz)\leq 0$ also in this case.
\\The proof of \eqref{eq:w-u-derivative} is thus complete.
\end{proof}

\begin{proof}[Proof of \cref{thm:talenti_rigidity}]
The first statement ($\ust = w$ in $[\bx, r_v]$) is a direct consequence of the monotonicity of $w-\ust$ (\cref{lem:monoton_w-u}), of the assumption $w(\bx)=\ust(\bx)$ and of the Talenti inequality $w-\ust\geq 0$ in $[0,r_v]$.

This also implies that $\mu (t) = \nu(t)$ for any $t\in (0, \ust(\bx))$, where $\nu$ is the distribution function of $w$. Hence, for any such $t$, equality holds in \cref{eq:main_thm0}. In particular,  the super-level set $\brcs*{\abs{u}>t}$ satisfies $\mcI_{N-1,N}(\meas\pths*{\brcs*{\abs{u}>t}})=\Per\pths*{\brcs*{\abs{u}>t}}$. By the rigidity in the Lévy-Gromov inequality, this implies that $\Xdm$ is a spherical suspension.

Assume now $\bx=0$ (thus $u^\star = w$ in $[0,r_v]$), $u\in\Lip(\Omega)$ and $\mwug{u}\neq 0$ $\meas$-almost everywhere in $\spt (u)$. Putting together the gradient comparison inequality \eqref{eq:GradComp} (with $q=2$) and the \PSz{} inequality (\cref{eq:psz}), we find
	\begin{equation*}
	\int_0^{r_1}\abs{\nabla u^\star}^2 \dmNN \leq \int_\Omega \abs{\nabla u}^2 \dmeas	\leq \int_0^{r_1}\abs{\nabla w}^2 \dmNN.
	\end{equation*}
The equality assumption, however, implies that the first and the last expressions coincide: thus, equality in the \PSz{} inequality is achieved. By rigidity in tge \PSz\ inequality, then, $u$ is radial.
\end{proof}
\let\bx\relax\let\bs\relax
\let\ust\relax\let\ush\relax\let\wsh\relax

\subsection{Stability}
\label{subsec:alm_rig}

In this Section, we will prove a stable version of the rigidity result (\cref{thm:talenti_rigidity}); we only consider the case where $\Eform = \Ch$, so that $\Eop$ is the Laplacian. We first need to recall some results on the convergence of \mms{}s and of functions defined therein.

\begin{assumption}\label{assumption:pmgh_conv} From now on, the following assumptions will be made:

\emph{Spaces:} 	$\brcs*{\XX_i}_{i\in\N}=\brcs{\Xdxmi}_{i\in \N}$ 
				and $\XX=\Xdxm$ 
will be pointed \mms{}s satisfying the $\RCD(N-1,N)$ condition for some $N\geq 2$, with $\meas_i(\Xmms_i)=1$, $\meas (\Xmms)=1$. 

\emph{Convergence of spaces:}  we will assume that  $\XX_i$ converge in the $\pmGH$ sense to $\XX$; by \cite[Section 3.5]{GigliMondinoSavare},  $\pmGH$ convergence coincides in our setting with $\pmG$ convergence; thus we can assume that the following conditions hold:
	\begin{enumerate}
	\item[(GH1)]\label{condition:GH1} $\Xmms_i$ and $\Xmms$ are all contained in a common metric space $(\Ymms,\dist)$, with $\dist_i=\restr{\dist}{\Xmms_i\times\Xmms_i}$, and $x_i \to x$;
	\item[(GH2)]\label{condition:GH2} $\spt \meas_i=\Xmms_i$ and $\spt \meas = \Xmms$;
	\item[(GH3)]\label{condition:GH3} The measures $\meas_i$ narrowly converge to $\meas$:
		\begin{equation*}
		\lim_{i\to \infty}\int_{\Ymms} \phi \dmeas_i = \int_{\Ymms} \phi \dmeas \quad \text{for all $\phi \in \Cb(\Ymms)$}, 
		\end{equation*}
		where $\Cb(\Ymms)$ is the space of continuous and bounded functions on $(\Ymms, \dist)$.
	\end{enumerate}
\end{assumption}

\begin{remark}[Compactness and stability of $\RCD(K,N)$ sequences]\label{remark:comp_RCD} Fix $K\in \R$ and $N\in (1,\infty)$. Every sequence $\brcs*{\XX_i}_{i\in\N}=\brcs{\Xdxmi}_{i\in \N}$ of pointed $\RCD(K,N)$ spaces admits a subsequence which converges in the $\pmGH$ sense to a pointed \mms{} $\XX$, and $\XX$ itself satisfies an $\RCD(K,N)$ condition. 
Indeed:
	\begin{itemize}
	\item relative compactness follows (as in the classical Gromov's precompactness Theorem) from \cite[Proposition 5.2]{Gromov2007} and the Bishop-Gromov inequality (see \cite[Theorem 2.3]{Sturm2006b} and \cite[Section 5.4]{LottVillani}); 
	\item the fact that the class of $\RCD(K,N)$ spaces is stable under $\mGH$ convergence follows from stability of the $\CD(K,N)$ class (see \cite[Section 5.3]{LottVillani}), from the stability of the $\RCD(K,\infty)$ class under $\pmG$ convergence (see \cite[Theorem 7.2]{GigliMondinoSavare}) and from the equivalence of $\pmG$ and $\pmGH$ convergence for $\RCD(K,N)$ spaces \cite[Section 3.5]{GigliMondinoSavare}. 
	\end{itemize}		
\end{remark}
	
\begin{remark}[$L^2$ functions]\label{remark:L2functions}
Assume that $\ball{x_i}{R_i}$ and $\ball{x}{R}$ are metric balls in $\Xmms_i$ and $\Xmms$ respectively.
Let $f_i \in L^2 \pths*{\ball{x_i}{R_i},\meas_i}$ and $f \in L^2 \pths*{\ball{x}{R},\meas}$ be $L^2$ functions on such balls; by extending such functions to be $0$ out of the balls on which they are defined, we can equivalently assume  $f_i \in L^2 \pths*{\Xmms_i,\meas_i}$ and $f \in L^2 \pths*{\Xmms,\meas}$; by the assumption that the spaces $\Xmms_i$ and $\Xmms$ are contained in $\Ymms$, up to a further extension we actually have $f_i \in L^2 \pths*{\Ymms,\meas_i}$ and $f \in L^2 \pths*{\Ymms,\meas}$. 
\end{remark}

\begin{definition}[Convergence of $L^2$ functions]
Let $f_i \in L^2 \pths*{\ball{x_i}{R_i},\meas_i}$ and $f \in L^2 \pths*{\ball{x}{R},\meas}$ as in \cref{remark:L2functions}. Following \cite[Definition 6.1]{GigliMondinoSavare}, we say that:
	\begin{enumerate}[(a)]
	\item $f_i \rightharpoonup f$ in the \emph{weak} $L^2$ sense if 
		\begin{gather*}
		\lim_{i\to \infty}\int_{\Ymms} \phi f_i \dmeas_i = \int_{\Ymms} \phi f \dmeas \quad \text{for all $\phi \in \Cb(\Ymms)$} \\
		\sup_i \norm*{f_i}_{L^2 \pths*{\ball{x_i}{R_i},\meas_i}} <\infty.
		\end{gather*}
	\item $f_i \rightarrow f$ in the \emph{strong} $L^2$ sense if, in addition,
		\begin{align*}
		\lim_{i\to \infty} \norm*{f_i}_{L^2 \pths*{\ball{x_i}{R_i},\meas_i}} &= \norm*{f}_{L^2 \pths*{\ball{x}{R},\meas}}.
		\end{align*}
	\end{enumerate}
\end{definition}

In order to obtain the stability result, we establish a series of auxiliary lemmas of independent interest.  
We start by showing that $L^2$-strong convergence of maps implies the pointwise convergence of the distribution functions to the distribution function of the limit. 

\begin{lemma}[Convergence of distribution functions]\label{lem:converg_of_distr}
Let $\XX_i\overset{{\rm pmGH}}{\longrightarrow}\XX$ be \pmms{}s satisfying \cref{assumption:pmgh_conv}. Let $\ball{x_i}{R_i}$ and $\ball{x}{R}$ be metric balls in $\Xmms_i$ and $\Xmms$ respectively, and let $f_i \in L^2 \pths*{\ball{x_i}{R_i},\meas_i}$ and $f \in L^2 \pths*{\ball{x}{R},\meas}$. Assume $\mu_i\doteq \mu_{f_i}$ and $\mu\doteq \mu_f$ are the distribution functions of $f_i$ and $f$ respectively. If $f_i \to f$ $L^2$-strongly, then $\mu_i(t)$ converges to $\mu(t)$ for every $t\in (0,+\infty)\setminus C$, where $C$ is a countable set.
\end{lemma}

\begin{proof}
Let us fix $t\in(0,+\infty)$. We need to show that (except for a countable number of such $t$)
	\begin{equation}\label{eq:convsuplev}
		\limi \meas_i \pths*{\brcs*{{|f_i|}>t}}	= \meas \pths*{\brcs*{{|f|}>t}}.
	\end{equation}
Notice that 
	\begin{align*}
	\brcs*{x\in \Xmms_i \stset |f_i(x)|>t}&= \brcs*{x\in \Xmms_i \stset (x,|f_i(x)|)\in \Ymms\times (t,+\infty)}\\
	\brcs*{x\in \Xmms \stset |f(x)|>t}&= \brcs*{x\in \Xmms \stset (x,|f(x)|)\in \Ymms\times (t,+\infty)}.
	\end{align*}
Given a map $g:\Ymms \to \R$, we denote by $\ii \times g : \Ymms \to \Ymms \times \R$ the map 
	$
	\ii \times g(x)\doteq (x, g(x))
	$;
by the argument above, it holds that 
	\begin{align*}
	\brcs*{x\in \Xmms_i \stset |f_i(x)|>t}&= (\ii \times |f_i|)^{-1}\pths*{\Ymms \times (t,+\infty)}\\
	\brcs*{x\in \Xmms \stset |f(x)|>t}&= (\ii \times |f|)^{-1}\pths*{\Ymms \times (t,+\infty)}.
	\end{align*}
Define $\nu_i$ and $\nu$ to be the following push-forward measures on $\Ymms \times \R$ 
	\begin{equation*}
	\nu_i\doteq (\ii\times \abs{f_i})_\sharp \meas_i, \quad  \nu\doteq (\ii\times \abs{f})_\sharp \meas.
	\end{equation*}
Our goal (\cref{eq:convsuplev}) is equivalent to show that
	\begin{equation*}
	\limi \nu_i(\Ymms \times (t, +\infty)) = \nu(\Ymms \times (t, +\infty)).
	\end{equation*}
Notice that the topological boundary of $\Ymms \times (t, +\infty)$ is $\Ymms \times \brcs{t}$, which is $\nu$-negligible for all but a countable set of $t>0$ by the finiteness of $\meas$:
	\begin{equation*}
	\nu\pths*{\Ymms \times \brcs{t}}=\meas ((\ii \times |f|)^{-1}\pths*{\Ymms \times \brcs{t}})=\meas(\brcs{|f|=t}).
	\end{equation*}
Thus, it is sufficient to show that the measures $\nu_i$ converge narrowly to $\nu$ in $\Ymms\times \R$. To this aim, notice that  for every $\phi \in \Cb(\Ymms{ \times \R})$, one has
	$$
	\int_{\Ymms\times \R} \phi(x,s)\, \de \nu_i =\int_{\Ymms} \phi(x, |f_i(x)|)\,  \dmeas_i, \quad \int_{\Ymms\times \R} \phi(x,s)\, \de \nu =\int_{\Ymms} \phi(x, |f(x)|)\,  \dmeas.
	$$
Arguing as in  \cite[Theorem 5.4.4]{Ambrosio2014b} (see also \cite[Equation (6.6)]{GigliMondinoSavare}), one can show that the  items in the left converge to the one in the right. This proves the statement.
\end{proof}

In \cite[Theorem 4.2]{AmbrosioHonda}, a variant of the following \namecref{prop:comp_loc_sob} was established. The proof contained therein can be straightforwardly adapted to the present case. 

\begin{proposition}[Compactness of local Sobolev functions]\label{prop:comp_loc_sob}
Let  $\XX_i\overset{{\rm pmGH}}{\longrightarrow}\XX$ be \pmms{}s satisfying \cref{assumption:pmgh_conv}.
Let $R_i\to R$ be a convergent sequence of radii with $R_i, R >0$. 
Let $f_i \in \W (\ball{x_i}{R_i}, \dist, \meas_i)$ have bounded $\W$-norm: $\sup_i \norm*{f_i}_{\W}<+\infty$. 
Then  
there exists a function $f\in\W(\ball{x}{R},\dist, \meas)$ such that  $\brcs{f_i}_{i\in \N}$  converges $L^2$-strongly to $f$, up to a subsequence. 
\end{proposition}

The next step is to prove that $L^2$-strong convergence of functions with bounded $\W$-norms implies $L^2$-strong convergence of the symmetrizations.

\begin{lemma}\label{lem:conv_of_rearr}
Let $\XX_i, \XX, R_i, R, f_i$ satisfy the assumptions of \cref{prop:comp_loc_sob}, and 
let $f_i$ converge in the strong $L^2$ sense to $f\in \W(\ball{x}{R},\dist, \meas)$. 
Then, 
up to subsequences, the $f_i^\star$ converge to $f^\star$ in the strong $L^2(\JNN,\mNN)$ sense.
\end{lemma}

\begin{proof}
By \cref{prop:equimeas} and the \PSz{} inequality \eqref{eq:psz}, the $\W \pths*{\JNN,\deu,\mNN}$ norms of the functions $f_i^\star$ are bounded by $C\doteq \sup_i\norm*{f_i}_{\W \pths*{\ball{x_i}{R_i},\dist, \meas_i}}<\infty$, which implies that the $f_i^\star$ also converge (up to subsequences) to a function $g$ in the strong $L^2(\JNN,\mNN)$ sense. It remains to prove that $f^\star=g$ (at least $\mNN$-almost everywhere).

By \cref{lem:converg_of_distr}, the distribution functions $\mu_{f_i}$ converge pointwise to $\mu_f$ out of a countable set; similarly, $\mu_{f_i^\star}$ converge to $\mu_g$ out of a countable set. By equi-measurability of $f_i$ and $f_i^\star$, however, we have that $\mu_{f_i}=\mu_{f_i^\star}$, thus $\mu_f=\mu_g$ out of a countable set.  
Since both $\mu_f$ and $\mu_g$ are non-increasing and continuous, it follows that $\mu_f\equiv \mu_g$  and thus $f^\sharp\equiv g^\sharp$, which in turn implies $f^\star \equiv g^\star$.
Now $g$ was the $L^2$-limit of a sequence of non-increasing functions, thus it is non-increasing itself. By \cref{lem:decr_rearr}, we conclude that $f^\star=g^\star=g$ out of a countable set.
\end{proof}

In view of what we seek to achieve in \cref{lem:limit_sol}, we need the next elementary convergence result, which we shortly prove for the sake of completeness.

\begin{lemma}\label{lem:limit_vol}
Let $\XX_i\overset{{\rm pmGH}}{\longrightarrow}\XX$ be \pmms{}s satisfying \cref{assumption:pmgh_conv}.
Let $R_i \to R$ be such that $\meas_i \pths*{\ball{x_i}{R_i}} = v \in (0,1)$ for all $i\in \N$. Then $\meas (\ball{x}{R}) = v$.
\end{lemma}

\begin{proof}
For any $\epsilon > 0$, the inclusions $\ball{x}{R-\epsilon} \subset \ball{x_i}{R_i} \subset \ball{x}{ R+\epsilon}$ hold for $i$ large enough. Thus, by weak convergence of the measures, we have for any $\epsilon>0$:
	$$
	\meas \pths*{\ball{x}{R-\epsilon}} \leq \liminf \meas_i \pths*{\ball{x_i}{R_i}} = v,\quad  \meas \pths*{\ball{x}{R+\epsilon}} \geq \limsup \meas_i \pths*{\ball{x_i}{R_i}} = v.
	$$
Moreover, the following holds (because the space is length):
	\begin{equation*}
	\bigcup_{\epsilon >0} \ball{x}{R-\epsilon} = \ball{x}{R} \subset \overline{\ball{x}{R}} = \bigcap_{\epsilon >0} \ball{x}{R+\epsilon}.
	\end{equation*}
Combining these two facts, and the fact that $\meas \pths*{\partial \ball{x}{R}}=0$ for every $R>0$ (which is true on $\RCD(K,N)$ spaces), implies the statement.
\end{proof}

The next \namecref{lem:limit_sol} analyses the convergence of solutions to the Poisson problem.

\begin{lemma}\label{lem:limit_sol}
Let $\XX_i\overset{{\rm pmGH}}{\longrightarrow}\XX$  be \pmms{}s satisfying \cref{assumption:pmgh_conv}.
Let $R_i \to R$ be such that $\meas_i \pths*{\ball{x_i}{R_i}} = v \in (0,1)$ for all $i\in \N$. 
Let $f_i \in \W \pths*{\ball{x_i}{R_i},\dist, \meas_i}$ with 
	\begin{equation*}
	\sup_i\norm*{f_i}_{W^{1,2} \pths*{\ball{x_i}{R_i},\dist, \meas_i}}<\infty.
	\end{equation*}
Assume that $u_i\in W^{1,2} \pths*{\ball{x_i}{R_i},\dist, \meas_i} $ are weak solutions to
	\begin{equation*}
	\begin{cases}
	-\Delta u_i = f_i 	& \text{in $\ball{x_i}{R_i}$}\\
	u_i=0				& \text{on $\partial \ball{x_i}{R_i}$}
	\end{cases}
	\end{equation*}
and $w_i \in W^{1,2}(\JNN,\deu,\mNN)$ are weak solutions to
	\begin{equation*}
	\begin{cases}
	-\lapNN w_i = f_i ^\star	& \text{in $[0,r_{v})$}\\
	w_i=0						& \text{at $r_{v}$}
	\end{cases}
	\end{equation*}
where $r_{v}\doteq \HNN^{-1}(v)$.
Then, up to extracting a subsequence:
	\begin{enumerate}[(i)]
	\item $f_i$ converges in $L^2$-strong to a function $f\in L^2(\ball{x}{R})$ with $\meas\pths{\ball{x}{R}} = v$; $f_i^\star$ converges in $L^2$-strong to $f^\star$;
	\item $u_i$ converges in $L^2$-strong to a weak solution $u$ of
			\begin{equation}\label{eq:poiss_i}
			\begin{cases}
			-\Delta u = f 	& \text{in $\ball{x}{R}$}\\
			u=0				& \text{on $\partial \ball{x}{R}$}
			\end{cases}
			\end{equation}
	\item $w_i$ converges in $L^2$-strong to a weak solution $w$ of 
		\begin{equation*}
		\begin{cases}
		-\lapNN w = f ^\star	& \text{in $(0,r_v)$}\\
		w=0						& \text{at $r_v$}.
		\end{cases}
		\end{equation*}
	\end{enumerate}
\end{lemma}

\begin{proof}
Assertion \textit{(i)} is granted by \cref{lem:conv_of_rearr} and \cref{lem:limit_vol}.
In \cref{eq:gradients_proof_4}, the following identity was proved:
	\begin{equation*}
	\int_0^{r_{v}} \abs{w_i'(\rho)}^2 \dmNN = \int_0^{v} \pths*{\frac{F_i(\xi)}{ \hNN(\HNN^{-1}(\xi))}}^2 \,\de\xi,
	\end{equation*}
where as usual $F_i(\xi)\doteq \int_0^\xi f_i^\sharp (t)\de t$. 
Notice that for any $\xi\in(0,v)$
	\begin{equation*}
	F_i(\xi)^2 = 
		\pths*{\int_0^\xi 1\cdot f_i^\sharp(t)\, \de t}^2 
		\leq \pths*{\norm*{1}_{L^2(0,\xi)}\norm[\big]{f_i^\sharp}_{L^2(0,\xi)}}^2
		\leq \xi \norm*{f_i}^2_{L^2(\ball{x_i}{R_i})}\leq C^2\xi,
	\end{equation*}
where $C\doteq \sup_i\norm*{f_i}_{W^{1,2} \pths*{\ball{x_i}{R_i},\dist, \meas_i}}$.
Thus we have:
	\begin{equation*}
	\norm*{w_i'}_{L^2((0,r_{v}),\mNN)} \leq 
		C^2 c_{N-1,N}^2\int_0^{v}\frac{\xi \, \de \xi}{\sin^{2N-2}(\HNN^{-1}(\xi))},
	\end{equation*}
where $c_{N-1,N}>0$ is the constant appearing in the definition of $\hNN$.
Since $\HNN(\xi)$ is of the same order as $\xi\mapsto \xi^N$ near $0$, the integrand at the right hand side is asymptotic to $\xi^{1-\frac{2N-2}{N}}=\xi^{-1+\frac{ 2}{N}}$ when $\xi\to 0$. In particular, the integral is finite and only depends on $N$ and $v$: the $L^{2}$-norm of $w'_i$ is thus uniformly bounded:
	\begin{equation*}
\norm*{w_i'}_{L^2((0,r_{v}),\mNN)}  \leq  C^2 c_{N-1,N}^2 \kappa(N,v).
	\end{equation*}
By Poincaré inequality, $\norm{w_i}_{W^{1,2}([0,r_1),\deu,\mNN)}$ are also uniformly bounded. Using  the Talenti-type \cref{theorem:talenti} with the associated gradient comparison \eqref{eq:GradComp},  we infer that $\norm{u_i}_{W^{1,2}(\ball{x_i}{R_i},\dist,\meas_i)}$ are uniformly bounded as well. Thus,  by \cref{prop:comp_loc_sob}  (and up to subsequences), the $u_i$'s converge in $L^2$-strong to a function $u$ and the $w_i$'s converge in $L^2$ strong to a function $w$; moreover,  by \cref{lem:conv_of_rearr}, $u_i^\star$ converges in $L^2$ strong to $u^\star$.

In order to prove point \textit{(ii)} (and, analogously, point \textit{(iii)}), we apply \cite[Corollary 4.3]{AmbrosioHonda}.  
To this aim, observe that  (up to subsequences) we can assume that $u_{i}$ converges to $u$ also weakly in $W^{1,2}$ by \cite[Proposition 3.1]{AmbrosioHonda}. Moreover,  every  $\psi\in \W_0\pths*{\ball{x}{R},\dist,\meas}$ can be recovered as the strong $W^{1,2}$ limit of a sequence of functions $\psi_i \in \W_0\pths*{\ball{x_i}{R_i},\dist,\meas_i}$ by  \cite[Lemma 2.10]{AmbrosioHonda}. Therefore we have:
	\begin{itemize}
	\item $\Ch(u_i,\psi_i)= \int_{\Ymms} f_i \psi_i \dmeas_i$ by the definition of $u_i$ as a weak solution of the Poisson problem;
	\item $\limi \Ch (u_i, \psi_i) = \Ch (u, \psi)$ by \cite[Corollary 4.3]{AmbrosioHonda};
	\item $\limi \int_{\Ymms} f_i \psi_i \, \dmeas_i = \int_{\Ymms} f \psi \, \dmeas$ by \cite[Equation (6.7)]{GigliMondinoSavare}.
	\end{itemize}
In particular,
	\begin{equation*}
	\int_{\Ymms} f \psi \dmeas = \Ch (u, \psi),
	\end{equation*}
thus $u$ is a weak solution of \cref{eq:poiss_i}. An analogous argument proves statement \textit{(iii)}.
\end{proof}

We finally have the tools to prove a stability result, by considering a contradicting sequence, applying a compactness argument, and exploiting the already proven rigidity result on the limit space.

\begin{theorem}[Stability in the Talenti-type theorem] \label{thm:stabilityTalenti}
For every $\epsilon>0, N\in [2,\infty), v\in (0,1), 0<c_{l}\leq  c_{u}<\infty$ there exists $\delta = \delta(\epsilon,N,v, \frac{c_l}{c_u})>0$ such that the following statement holds. Assume:
	\begin{enumerate}[(i)]
	\item $\Xdm$ is a $\RCD(N-1 , N)$ metric space  with $\meas(\Xmms)=1$ and $\Omega=\ball{x}{R}\subset \Xmms$ is an open ball with $\meas (\Omega) = v$;
	\item $f\in \W(\Omega, \dist, \meas)$ with $c_l \leq  \norm*{f}_{L^2\Odm} \leq \norm*{f}_{\W(\Omega,\dist,\meas)}\leq c_u$;
	\item $u\in W^{1,2}_0(\Omega,\dist,\meas)$ weakly solves $-\Delta u=f$;
	\item $w\in \W([0,r_v),\deu,\mNN)$ weakly solves $- \lapNN w=f^\star$, with $w(r_{v})=0$ and $r_v\doteq \HNN^{-1}(v)$.
	\end{enumerate}
If $\norm{u^\star-w}_{L^2((0,r_v),\mNN)}<\delta$, then there exists a spherical suspension $\Zdm$ such that 
	\begin{equation*}
	\dmGH(\Xdm,\Zdm)<\epsilon.
	\end{equation*}
\end{theorem}

\begin{proof}
We argue by contradiction: assume there exist $\bar{\epsilon}, \bar{N}, \bar{v}, \bar{c}$ such that for any $i\in\N$ we can find an $\RCD(\bar{N}-1,\bar{N})$ space $(\Xmms_i,\dist_i,\meas_i)$, a ball $\Omega_i=\ball{x_i}{R_i}\subset \Xmms_i$,  and functions $f_i\in \W(\Omega_i,\dist_i,\meas_i)$, $u_i\in \W_0(\Omega_i,\dist_i,\meas_i)$, $w_i\in W^{1,2}([0,r_v),\deu,\mNN)$ such that for any $i\in \N$
	\begin{gather*}
	-\Delta u_i  =f_i	 \text{ weakly,}	\quad - \lapNN w_i=f_i^\star \text{ weakly,  with $w_{i}(r_{v})=0$,}\\
	\bar{c} \leq  \norm*{f}_{L^2\Odm} \leq \norm*{f}_{\W(\Omega,\dist,\meas)}\leq 1, \quad \norm*{u_i^\star - w_i}_{L^2((0,r_v),\mNN)}<\frac{1}{i},
	\end{gather*}
and moreover
	\begin{equation}\label{eq:dist_sph_susp}
		\inf \brcs[\Big]{\dmGH\pths*{(\Xmms_i,\dist_i,\meas_i), \Zdm}
		\;\Big|\;		
		\text{$\Zdm$ is a spherical suspension}} \geq \bar{\epsilon}.
	\end{equation}
Up to subsequences, we can assume that:
	\begin{itemize}
	\item $\Xdxmi$ converge to an $\RCD(N-1,N)$ space $\Xdxm$, and \cref{assumption:pmgh_conv} is satisfied (see \cref{remark:comp_RCD}); moreover, $\meas(B_{R}(x))=v$ by \cref{lem:limit_vol};
	\item $f_i$, $u_i$ and $w_i$ satisfy the conclusions of \cref{lem:limit_sol}: that is, $f_i$ converges in $L^2$-strong to a function $f\in L^2(\ball{x}{R})$; $f_i^\star$ converges in $L^2$-strong to $f^\star$; $u_i$ converges in $L^2$-strong to a weak solution $u$ of $-\Delta u = f$  in $\ball{x}{R}$ (with zero boundary condition); $w_i$ converges in $L^2$-strong to a weak solution $w$ of $-\lapNN w = f ^\star$	 in $[0,r_{v})$ with $w(r_{v})=0$.
	\end{itemize}
Notice that 
	\begin{equation}\label{eq:XdistSS}
	\dmGH\pths*{(\Xmms,\dist,\meas), \Zdm}\geq \bar{\epsilon}
	\end{equation}
for any spherical suspension $\Zdm$, by \cref{eq:dist_sph_susp}. 
However, by the $L^{2}$-strong convergence of $u_i^\star$ to $u^\star$ (\cref{lem:conv_of_rearr}) and the $L^{2}$-strong convergence of $w_i$ to $w$, one has
	\begin{equation*}
	\norm*{u^\star - w}_{L^2([0,r_v), \mNN)}= \limi \norm*{u_i^\star - w_i}_{L^2([0,r_v), \mNN)}=0,
	\end{equation*}
which implies that $u^\star = w$. Moreover, since $f$ is the $L^{2}$-strong limit of the $f_i$'s, it has $L^2$-norm bounded from below by $\bar{c}$, thus it is different from $0$ on a non-negligible set. By the rigidity in the Talenti-type comparison (\cref{thm:talenti_rigidity}),  $\Xdm$ needs to be a spherical suspension, contradicting \eqref{eq:XdistSS}.
\end{proof}

\begin{corollary}
For every $\epsilon>0, N\in [2,\infty), v\in (0,1), 0<c_{l}\leq  c_{u}<\infty$ there exists $\delta_1 = \delta_1(\epsilon,N,v, \frac{c_l}{c_u})>0$ such that the following statement holds. Assume that the conditions \textit{(i)}-\textit{(iv)} of \cref{thm:stabilityTalenti} hold. If $w(0)-u^\star(0)<\delta_1$, then there exists a spherical suspension $\Zdm$ such that 
	\begin{equation*}
	\dmGH(\Xdm,\Zdm)<\epsilon.
	\end{equation*}
\end{corollary}

\begin{proof}
By \cref{lem:monoton_w-u}, $w-u^\star$ is non-increasing (and non-negative) in $[0, r_v]$. Thus 
	\begin{equation}\notag
	\norm{u^\star-w}_{L^2((0,r_v),\mNN)}\leq \pths*{w(0)-u^\star(0)}\, \sqrt{v}.
	\end{equation}
The result follows from \cref{thm:stabilityTalenti} with $\delta_1=\frac{\delta}{\sqrt{v}}$.
\end{proof}


\section{Applications}

\subsection{Improved Sobolev embeddings}
\label{sec:sob_emb}
As a first application of the Talenti-type comparison \cref{theorem:talenti}, we deduce a  series of Sobolev-type inequalities that to best of our knowledge are new in the framework of $\RCD(K,N)$ spaces (compare with  \cite[Section 3.3]{Kesavan2006} for the Euclidean setting).

\begin{theorem}\label{thm:sob_emb}
Let $\Xdm$ be an $\RCD(K,N)$ space for some $K>0$, $N\in (1,\infty)$, with  $\meas(\Xmms)=1$. Let  $\Omega \subset \Xmms$ be an open domain with  measure $v\doteq \meas(\Omega)\in (0,1)$. Let  $u:\Omega \to \R$ be a function in $W^{1,2}_0(\Omega)$ and  $f\in L^2(\Omega,\meas)$. Assume that $u$ is a weak solution to the equation $-\Eop (u)=f$, where $\Eform$ is an $\alpha$-uniformly elliptic bilinear form  as in \cref{assumption:eform}. Then the following statements hold:
	\begin{enumerate}
	\item If $f\in L^p(\Omega,\meas)$ with $\frac{N}{2}<p \leq \infty$, then $u\in L^{\infty} (\Omega,\meas)$ and 
		\begin{equation}\label{eq:sob_emb1}
		\norm*{u}_{L^\infty(\Omega,\meas)}\leq \frac{c_1(K,N,v,p)}{\alpha} \norm*{f}_{L^p(\Omega,\meas)},
		\text{ with }
		c_1(K,N,v,p)\doteq \int_0^{v} \frac{\xi^{1-\frac{1}{p}}}{\hKN(\HKN^{-1}(\xi))^2}\,\de \xi <\infty,
		\end{equation}
		where we adopt the convention that $\frac{1}{p}=0$ if $p=\infty$.
	\item If $f\in L^p(\Omega,\meas)$ with $2\leq p \leq \frac{N}{2}$, and $q\geq 1$ is such that $q\pths*{\frac{1}{p}-\frac{2}{N}}<1$, then $u\in L^{q} (\Omega,\meas)$ and 
		\begin{equation*}
		\begin{gathered}
		\norm*{u}_{L^q(\Omega,\meas)}\leq \frac{c_2(K,N, v, p, q)}{\alpha} \norm{f}_{L^p(\Omega,\meas)},\\
		\text{with} \quad
		c_2(K,N, v, p, q)\doteq \pths*{\int_0^v \pths*{\int_s^v \frac{\xi^{1-\frac{1}{p}}}{\hKN(\HKN^{-1}(\xi))^2}\,\de \xi}^q\de s }^{\frac{1}{q}} <\infty.
		\end{gathered}
		\end{equation*}
	\end{enumerate}
\end{theorem}

\begin{proof}
By \cref{theorem:talenti}, $u^\sharp$ satisfies the following inequality (see \cref{eq:main_thm2}):
	\begin{equation*}
	0 \leq u^\sharp (s) \leq 
		\frac{1}{\alpha}\int_s^{\meas(\Omega)} \frac{1}{\isoKN(\xi)^2} \int_0^\xi f^\sharp (t) \,\de t\,\de \xi,
		\qquad \forall s\in (0,\meas(\Omega)).
	\end{equation*}
If $f\in L^p(\Omega,\meas)$ for some $p\in[2,\infty]$, then by \Holder{} inequality, by equimeasurability of $f$ and $f^\sharp$ (\cref{prop:equimeas}), and by the characterization of the isoperimetric profile on $\JKN$ (\cref{lemma:isoprof_model}),
	\begin{equation}\label{eq:sob_emb_pr_1}
	u^\sharp (s) \leq  
		\frac{1}{\alpha} \norm*{f}_{L^p(\Omega,\meas)} 
		\int_s^{\meas(\Omega)}\frac{\xi^{1-\frac{1}{p}}}{\hKN(\HKN^{-1}(\xi))^2}\,\de \xi,
	\end{equation}
with the convention that $\frac{1}{p}=0$ if $p=\infty$. 
Now by the estimates on $\hKN$ (\cref{lem:asympt}) there exist constants $C_0>0$ and $C_1>0$ only depending on $K>0$, $N\in (1,\infty)$ and $v=\meas(\Omega)\in (0,1)$ such that for all $\xi\in[0,\meas(\Omega)]$
	\begin{equation}\label{eq:Esth>}
	\hKN(\HKN^{-1}(\xi)) \geq C_0 (\HKN^{-1}(\xi))^{N-1}\geq  C_1 \xi^{\frac{N-1}{N}}.
	\end{equation}
We can thus draw the following conclusions:

 \proofstep{Case 1:} If $\frac{N}{2}<p \leq \infty$, then 
	\begin{equation*}
	\int_s^{\meas(\Omega)} \frac{\xi^{1-\frac{1}{p}}}{\hKN(\HKN^{-1}(\xi))^2}\,\de \xi \leq 
		\frac{1}{C_1^2}\int_0^{\meas(\Omega)} \xi^{\frac{2}{N}-\frac{1}{p}-1} \,\de \xi = 
		\frac{v^{\frac{2}{N}-\frac{1}{p}}}{\pths*{\frac{2}{N}-\frac{1}{p}}C_1^2}
		<\infty, \quad \forall s\in [0,v].
	\end{equation*}
 By \cref{eq:sob_emb_pr_1} and by equimeasurability of $u$ and $u^\sharp$, this implies \cref{eq:sob_emb1}.

\proofstep{Case 2:} If $p=\frac{N}{2}$ and $q\geq 1$, then
	\begin{equation*}
	\pths*{\int_s^{\meas(\Omega)} \frac{\xi^{1-\frac{1}{p}}}{\hKN(\HKN^{-1}(\xi))^2}\,\de \xi}^q \leq 
		\frac{1}{C_1^{2q}}\pths*{\log v - \log s}^q
	\end{equation*}
	and thus
	\begin{equation*}
	\begin{split}
	\norm{u}^q_{L^q(\Omega,\meas)}=\norm{u^\sharp}^q_{L^q((0,v))}
	&\leq \frac{\norm*{f}^q_{L^p(\Omega,\meas)}}{ \alpha^q} \int_0^v \pths*{\int_s^v \frac{\xi^{1-\frac{1}{p}}}{\hKN(\HKN^{-1}(\xi))^2}\,\de \xi}^q\de s \\
	&\leq  \frac{\norm*{f}^q_{L^p(\Omega,\meas)}}{ C_1^{2q} \alpha^q} \int_0^1 (-\log s)^q \, \de s <\infty.
	\end{split}
	\end{equation*}

\proofstep{Case 3:} If $2\leq p < \frac{N}{2}$ and $q\geq 1$, with $q\pths*{\frac{1}{p}-\frac{2}{N}}<1$, then
	\begin{equation*}
	\pths*{\int_s^{\meas(\Omega)} \frac{\xi^{1-\frac{1}{p}}}{\hKN(\HKN^{-1}(\xi))^2}\,\de \xi}^q \leq \frac{1}{C_1^{2q}\pths*{\frac{1}{p}-\frac{2}{N}}^q}\pths*{s^{\frac{2}{N}-\frac{1}{p}}-v^{\frac{2}{N}-\frac{1}{p}}}^q
	\end{equation*}
and thus 
	\begin{equation*}
	\begin{split}
	\norm{u}^q_{L^q(\Omega,\meas)}=\norm{u^\sharp}^q_{L^q((0,v))}
	&\leq \frac{\norm*{f}^q_{L^p(\Omega,\meas)}}{ \alpha^q} \int_0^v \pths*{\int_s^v \frac{\xi^{1-\frac{1}{p}}}{\hKN(\HKN^{-1}(\xi))^2}\,\de \xi}^q\de s \\
	&\leq \frac{\norm*{f}^q_{L^p(\Omega,\meas)}}{C_1^{2q} \alpha^q  \pths*{\frac{1}{p}-\frac{2}{N}}^q } \int_0^v \pths*{s^{\frac{2}{N}-\frac{1}{p}}-v^{\frac{2}{N}-\frac{1}{p}}}^q  \,\de s <\infty.
	\end{split}
	\end{equation*}
\end{proof}

Let now $2\leq p \leq \infty$. If we define $D_{\Omega,p} (\Eop)$ to be the space
	\begin{equation*}
	D_{\Omega,p} (\Eop)\doteq \brcs*{u\in D_\Omega(\Eop)\stset \Eop (u)\in L^p(\Omega,\meas)},
	\end{equation*}
where $D_\Omega (\Eop)$ is the space defined in \cref{def:domain_of_LE}, then 
\cref{thm:sob_emb} can be restated as follows:

\begin{corollary}[Improved Sobolev embeddings]\label{Cor:SobolevIneq}
Let $\Xdm$ be an $\RCD(K,N)$ space for some $K>0$, $N\in (1,\infty)$,  with $\meas(\Xmms)=1$. Let $\Omega \subset \Xmms$ be an open domain with  measure $v\doteq \meas(\Omega)\in (0,1)$ and $u:\Omega \to \R$ be a function in $W^{1,2}_0(\Omega)$. Let $\Eform$ be a bilinear form satisfying \cref{assumption:eform} with uniform ellipticity parameter $\alpha$.
	\begin{enumerate}[(a)]
	\item If $\frac{N}{2}< p \leq \infty$, then $W^{1,2}_0(\Omega)\cap D_{\Omega,p}(\Eop)\subset L^{\infty}(\Omega)$ with
		\begin{equation*}
		\norm{u}_{L^{\infty}(\Omega,\meas)}\leq C(K,N,v,p,\alpha) \norm{\Eop(u)}_{L^p(\Omega,\meas)}.
		\end{equation*}
	\item If $2\leq p \leq \frac{N}{2}$ and $1\leq q < \pths*{\frac{1}{p}-\frac{2}{N}}^{-1}$, then $W^{1,2}_0(\Omega)\cap D_{\Omega,p}(\Eop)\subset L^{q}(\Omega)$ with
		\begin{equation*}
		\norm{u}_{L^{q}(\Omega,\meas)}\leq C(K,N,v,p,q,\alpha) \norm{\Eop(u)}_{L^p(\Omega,\meas)}.
		\end{equation*}		
	\end{enumerate}
\end{corollary}

\subsection{An \texorpdfstring{$\RCD$}{RCD} version of St.~Venant-P\'olya's torsional rigidity comparison theorem}\label{Sec:Torsion}

Given an open domain $\Omega\subset \R^n$, it is well known (see e.g. \cite{Evans2010}) that the Poisson boundary value problem
\begin{equation}\label{eq:PoisTR}
	\left\lbrace
	\begin{aligned}
	 - \Delta u & = 2 \quad \text{in $\Omega$}\\
	 u&=0   \quad \text{on $\partial \Omega$}
	\end{aligned}
	\right. 
\end{equation}
has a unique weak solution $u\in W^{1,2}_0(\Omega)$ which, by standard elliptic regularity, turns out to be of class $C^{\infty}(\Omega)$ (classical solution of class  $C^{\infty}(\bar{\Omega})$,  provided $\partial \Omega$ is $C^\infty$). Define also 
\begin{equation*}
T(\Omega):=\int_{\Omega} u(x) \, \de x.
\end{equation*}
When $n=2$ and $\Omega$ is simply connected, $u$ and $T(\Omega)$ are known in the literature as \emph{stress function}  and \emph{torsional rigidity} of $\Omega$, respectively (see \cite[p. 63]{Bandle1980} or \cite[Ch. 5.2]{PolyaSzego1951} for more details). For simplicity of notation we will keep using this terminology in general.  St.~Venant in 1856 conjectured that, among simply connected domains of a given volume $v\in (0,\infty)$,  the torsional rigidity  is maximized by the round ball.  Polya in 1948 \cite{Polya1948} settled the St.~Venant conjecture by proving more generally that for every bounded open set $\Omega\subset \R^{n}$ of volume $v\in (0,\infty)$, it  holds 
\begin{equation}\label{eq:TorRigComp}
T(\Omega)\leq T(B_{v})
\end{equation}
where $B_{v}$ is a Euclidean ball of  volume  $v$. For alternative proofs and related results, the interested reader may consult \cite[Ch. 5]{PolyaSzego1951}, \cite[p. 67]{Bandle1980}, \cite[Ch. 3.6]{Kesavan2006}, \cite[Ch. 5.7]{Baernstein2019}. An inequality in the spirit of \eqref{eq:TorRigComp}  was recently proved for smooth compact Riemannian manifolds with Ricci curvature bounded below  in \cite{Gamara2015}.

As an application of the techniques developed in this work, we establish the next far reaching extension to $\RCD$ spaces of the torsional rigidity comparison \eqref{eq:TorRigComp}. To this aim let us introduce a bit of notation. 
\\Given $v\in (0,1), K>0, N\in (1,\infty)$, let $\meas_{K,N}$ be as before and $r_{v}\in \pths[\big]{0,\pi \sqrt{\frac{N-1}{K}}}$ such that $\meas_{K,N}([0, r_{v}])=v$. Let $u_{K,N,v}$ be the weak solution (in the sense of \cref{def:poiss_model}) to the  Poisson problem
\begin{equation}\label{eq:PoisTRModKNv}
	\left\lbrace
	\begin{aligned}
	 - \Delta_{K,N} u_{K,N,v} & = 2 \quad \text{in $[0, r_{v})$}\\
	 u_{K,N,v}(r_{v})&=0  
	\end{aligned}
	\right. 
\end{equation}
and set
\begin{equation}\label{eq:defTKNv}
T_{K,N,v}:= \int_{[0, r_{v}]} u_{K,N,v} \, \dmKN.
\end{equation}

\begin{theorem}[$\RCD$ version of St.~Venant-P\'olya's  Theorem]\label{thm:StVenPol}
Let $\Xdm$ be an $\RCD(K,N)$ space for some $K>0$, $N\in (1,\infty)$,  with $\meas(\Xmms)=1$, and let $\Omega\subset \Xmms$ be an open domain with  measure $v\doteq \meas(\Omega)\in (0,1)$.  
\\Let $u\in W^{1,2}_{0}(\Omega)$ be the weak solution to the Poisson problem \eqref{eq:PoisTR} and set $T(\Omega)= \int_{\Omega} u \, \de \meas$.
Then
\begin{enumerate}
\item  {\rm Comparison}. $u\geq 0$ $\meas$-a.e. on $\Omega$ and  $u^{\star}\leq u_{K,N,v}$. Thus, in particular,  $T(\Omega)\leq T_{K,N,v}$.
\item  {\rm Rigidity}. Assume $N\geq 2$ and $K=N-1$. Then   $T(\Omega)= T_{N-1,N,v}$ if and only if $\Xdm$ is isomorphic as a \mms{} to a spherical suspension, i.e. there exists an $\RCD(N-2,N-1)$ space  $\Ydm$ with $\measY(\Ymms)=1$ such that $\Xdm$ is isomorphic as a \mms{} to $[0,\pi]\times^{N-1}_{\sin} \Ymms$.

\item  {\rm Stability}. 
Given $N\in [2,\infty), v\in (0,1), \varepsilon>0$ there exists $\delta=\delta(N,v,\varepsilon)>0$ such that if $K=N-1$ and $\Omega=B_{R}(x)$ is an open metric ball with $\meas(B_{R}(x))=v\in (0,1)$ satisfying $T(\Omega)\geq T_{N-1, N,v}-\delta$ then $\Xdm$ is $\varepsilon$-mGH close to a spherical suspension, i.e. there exists a spherical suspension $\Zdm$ such that 
	$
	\dmGH(\Xdm,\Zdm)<\epsilon.
	$

\end{enumerate}
\end{theorem}

\begin{proof}
First of all, notice that the weak solution $u\in W^{1,2}_{0}(\Omega, \dist, \meas)$ to the Poisson problem \eqref{eq:PoisTR} is the unique minimizer of the energy functional
$$
J(w):=\frac{1}{2} \int_{\Omega} |\nabla w|^{2} \, \de \meas - \int_{\Omega} 2 w\, \de \meas, \quad w\in W^{1,2}_{0}(\Omega).
$$ 
Since $J(u)\geq J(|u|)$, it follows that $u\geq 0$ $\meas$-a.e. on $\Omega$ and thus $u^{\star}\geq 0$.
\\The fact that $u^{\star}\leq  u_{K,N,v}$ $\meas_{K,N}$-a.e. on $[0,r_{v}]$  is a direct consequence of the Talenti-type  \cref{theorem:talenti}. Recalling \cref{prop:equimeas}, we infer that
\begin{equation*}
T(\Omega)=\int_{\Omega} u\,  \de \meas= \int_{[0, r_{v})} u^{\star} \, \dmKN \leq   \int_{[0, r_{v})} u_{K,N,v}\,  \dmKN=T_{K,N,v}.
\end{equation*}
Clearly, if $T(\Omega)=T_{K,N,v}$ then $u^{\star}=  u_{K,N,v}$ $\meas_{K,N}$-a.e.~on $[0,r_{v}]$ and we can apply  \cref{thm:talenti_rigidity} to infer that $\Xdm$ is isomorphic as a \mms{} to a spherical suspension (provided $N\in [2,\infty)$ and $K=N-1$).
\\We prove the last claim by contradiction. Assume that there exist $\bar{\epsilon}, \bar{N}, \bar{v}$ such that for any $i\in\N$ we can find an $\RCD(\bar{N}-1,\bar{N})$ space $(\Xmms_i,\dist_i,\meas_i)$, a ball $\Omega_i=\ball{x_i}{R_i}\subset \Xmms_i$,  and  $u_i\in \W_0(\Omega_i,\dist_i,\meas_i)$ such that for any $i\in \N$
	\begin{gather*}
	-\Delta u_i  = 2\quad \text{weakly,}\quad 	T_{N-1,N,v}-T(\Omega_{i}) <\frac{1}{i},
	\end{gather*}
and moreover
	\begin{equation}\label{eq:dist_sph_suspPol}
		\inf \brcs[\Big]{\dmGH\pths*{(\Xmms_i,\dist_i,\meas_i), \Zdm}
		\;\Big|\;		
		\text{$\Zdm$ is a spherical suspension}} \geq \bar{\epsilon}.
	\end{equation}
Up to subsequences, we can assume that $R_{i}\to R$ and that:
	\begin{itemize}
	\item $\Xdxmi$ converge to an $\RCD(N-1,N)$ space $\Xdxm$, and \cref{assumption:pmgh_conv} is satisfied (see \cref{remark:comp_RCD}); moreover, $\meas(B_{R}(x))=v$ by  \cref{lem:limit_vol};

	\item  $u_i$ satisfy the conclusions of \cref{lem:limit_sol}: that is, $u_i$ converges in $L^2$-strong to a weak solution $u\in \W_{0}(B_{R}(x))$ of $-\Delta u = 2$  in $\ball{x}{R}$ (with zero boundary condition).	 In particular, it follows that (see \cite[Eq. (6.7)]{GigliMondinoSavare}) 
	\begin{equation}\label{eq:LimitRigidity}
	T(\Omega)=\int_{\Omega} u\, \de \meas= \lim_{i\to \infty}  \int_{\Omega_{i}} u_{i}\, \de \meas_{i}=\lim_{i\to \infty} T(\Omega_{i}) =T_{N-1,N,v}.
	\end{equation}
	\end{itemize}
Notice that 
	\begin{equation}\label{eq:distXSSP}
	\dmGH\pths*{(\Xmms,\dist,\meas), \Zdm}\geq \bar{\epsilon}
	\end{equation}
for any spherical suspension $\Zdm$, by \cref{eq:dist_sph_suspPol}. However, combining \eqref{eq:LimitRigidity}  with the rigidity proved above in part 2, we infer that $\Xdm$ needs to be a spherical suspension, contradicting \eqref{eq:distXSSP}.
\end{proof}


\subsection{An alternative proof for the \texorpdfstring{$\RCD$}{RCD} version of  Rayleigh-Faber-Krahn-B\'erard-Meyer comparison theorem}

The next result was proved for the $p$-Laplacian by Mondino and Semola \cite{Mondino2019} in the more general setting of essentially non-branching $\CD(K,N)$ spaces (for $K>0$), as a consequence of a \PSz{} type inequality.
We give below an alternative proof in case $p=2$, based instead on Talenti's comparison theorem for $\RCD$ spaces. 

Firstly, we recall the notions of \textit{first eigenfunction} and \textit{first eigenvalue} of the Laplacian:

\begin{definition}
Let $\Omega \subset \Xmms$ be an open domain. For any non-zero function $w \in \W\Odm$ we define the Rayleigh quotient to be
	\begin{equation}
	\mathcal{R}_\Omega(w)\doteq \frac{\int_\Omega \mwug{w}^2 \dmeas}{\int_\Omega w^2 \dmeas}.
	\end{equation}
We say that:
	\begin{enumerate}[(i)]
	\item $\lambda_\Omega\doteq \inf \brcs*{\mathcal{R}_\Omega(w)\stset w\in\W_0(\Omega), w\not\equiv 0}$ is the first eigenvalue of the Laplacian in $\Omega$ with Dirichlet homogeneous conditions;
	\item $u \in \W_0(\Omega)$ is a first eigenfunction of the Laplacian in $\Omega$ (with Dirichlet homogeneous conditions) if it minimizes $\mathcal{R}_\Omega$ among functions $w\in\W_0(\Omega)$, $w\not\equiv 0$ (that is, $\mathcal{R}_\Omega(u)=\lambda_\Omega$).
	\end{enumerate}
When $\Xdm=\modelspace$, $v\in(0,1)$, and $\Omega=[0,\HKN^{-1}(v))$, we will denote the first eigenvalue with $\lambda_{K,N,v}$.
\end{definition}

\begin{theorem}
Let $\Xdm$ be an $\RCD(K,N)$ space for some $K>0$, $N\in (1,\infty)$, and let $\Omega\subset \Xmms$ be an open domain with measure $v\doteq \meas(\Omega)\in (0,1)$.
Then:
	\begin{enumerate}[(i)]
	\item  $\lambda(\Omega)\geq \lambda_{K,N,v}$;
	\item There exists a unique first eigenfunction of the Laplacian in $\Omega$, up to multiplication by a constant; such an eigenfunction can be chosen to be strictly positive and continuous in $\Omega$;
	\item If $u$ is a positive first eigenfunction, then $0<u^\star\leq w$ in $[0,r_v)$, where $r_v\doteq \HKN^{-1}(v)$ and $w$ is a solution to $-\lapKN w = \lambda_\Omega u^\star$ in $[0,r_v)$ with $w(r_v)=0$.
	\end{enumerate}
\end{theorem}

\begin{proof}
\proofstep{Step 1:} A first eigenfunction exists. This is standard and was already proved for example in \cite[Theorem 4.3]{Mondino2019}, but we recall here the argument: let $\brcs{u_n}_n$ be a minimizing sequence for $\mathcal{R}_\Omega$ with $u_n\in\W_0(\Omega)$, $\norm*{u_n}_{L^2(\Omega,\meas)}=1$ and $\int_\Omega \mwug{u}^2\dmeas \searrow \lambda_\Omega$. Since the embedding $W^{1,2}\Xdm \subset L^2 (\Xmms,\meas)$ is compact for an $\RCD(K,N)$ space with $K>0$, $N\in (1,\infty)$ (see \cite[Proposition 6.7]{GigliMondinoSavare}), the sequence $u_n$ converges to a function $u\in\W_0(\Omega)$ in the strong $L^2(\Omega,\meas)$ sense. Thus $\norm*{u}_{L^2(\Omega,\meas}=1$; moreover, from the very definition of $\lambda_\Omega$ and by the $L^2$-lower semicontinuity of the Cheeger energy, it holds that
	\begin{equation}
	\lambda_\Omega \leq \int_{\Omega}\mwug{u}^2 \dmeas \leq \liminf_{n\to \infty} \int_{\Omega}\mwug{u_n}^2 \dmeas = \lambda_\Omega.
	\end{equation}
Thus $u$ is a first eigenfunction.

\proofstep{Step 2:} Any first eigenfunction $u\in\W_0(\Omega)$ is a weak solution of: 
	\begin{equation}\label{eq:eigenval}
	\left\lbrace
	\begin{aligned}
	 - \Delta u & = \lambda_\Omega u \quad \text{in $\Omega$}\\
	 u&=0   \quad \text{on $\partial \Omega$}
	\end{aligned}
	\right.\quad .
	\end{equation}
This relies on a standard variational argument: for any $w\in\W_0(\Omega)$, we can explicitly compute the derivative of $\epsilon\mapsto \mathcal{R}_\Omega(u+\epsilon w)$ at $\epsilon=0$:
	\begin{equation}
	\restr{\frac{\de}{\de \epsilon}\mathcal{R}_\Omega(u+\epsilon w)}{\epsilon=0}
	=2\frac{\Ch(u,w)- \lambda_\Omega \int_\Omega u w \dmeas}{\norm*{u}_{L^2\pths{\Omega}}}.
	\end{equation}
Since this derivative must vanish, the identity $\Ch(u,w)= \lambda_\Omega \int_\Omega uw \dmeas$ needs to hold for any $w\in\W_0(\Omega)$, which proves that $u$ is a weak solution of  \eqref{eq:eigenval}.

\proofstep{Step 3:} Since $\RCD(K,N)$ spaces are locally doubling and Poincaré (see \cite{Sturm2006b, Rajala2012}), the Harnack-type results proved in \cite{Latvala} hold in these spaces: hence any first eigenfunction is continuous (Theorem 5.1 therein), and is strictly positive in $\Omega$ up to multiplying by a constant (Corollary 5.7 and Corollary 5.8 therein). The same results also imply the uniqueness of the first eigenfunction up to a multiplicative constant: if $u_1$ and $u_2$ are two first eigenfunctions with $\frac{u_1}{u_2}$ non-constant, then there exists $\gamma>0$ such that $u_1-\gamma u_2$ is a first eigenfunction that changes sign in $\Omega$.

\proofstep{Step 4:} Let now $w$ be a solution to $-\lapKN w = \lambda_\Omega u^\star$ in $[0,r_v)$ with $w(r_v)=0$. By the definition of $w$ and by using $w$ itself as a test function, it holds that $\int_0^{r_v} \mwug{w}^2 \dmKN=\lambda_\Omega \int_0^{r_v} u^\star w \dmKN$; by the Talenti-type theorem it holds that $0<u^\star \leq w$.
Thus 
	\begin{equation}
	\lambda_{K,N,v}\leq \frac{\int_0^{r_v} \mwug{w}^2 \dmKN}{\int_0^{r_v} w^2 \dmKN} = \lambda_\Omega\frac{\int_0^{r_v} u^\star w \dmKN}{\int_0^{r_v} w^2 \dmKN}\leq \lambda_\Omega\frac{\int_0^{r_v} w^2 \dmKN}{\int_0^{r_v} w^2 \dmKN}=\lambda_\Omega.
	\end{equation}
\end{proof}


\section{Appendix: the case of a smooth Riemannian manifold with positive Ricci curvature} 
Since some of the results of the paper seem to be new even in the setting of smooth Riemannian manifolds (compare with \cite{CLM}),  in this appendix we briefly give the corresponding smooth statements without the technicalities of $\RCD(K,N)$ spaces. In this way, our aim is to make the results accessible to a more general audience. \par
\medskip
Let $(M,g)$ be a complete $N$-dimensional Riemannian manifold, $N\geq 2$, with Ricci curvature tensor satisfying ${\rm Ric}_{g}\geq K \,g$ for some constant $K>0$. 
By Bonnet-Myers theorem, $M$ must be compact. Denote with ${\rm vol}_{g}$ the Riemannian volume measure and with $\meas_{g}\doteq {\rm vol}_{g}(M)^{-1} \, {\rm vol}_{g}$ the associated normalized measure.  Let $h\in \Gamma_{meas}( T^{*}M {\otimes} T^{*}M)$ be a symmetric bilinear form on $M$ with measurable coefficients and assume there exists $\alpha,\beta>0$ such that 
\begin{equation}\label{eq:Aalphabeta}
 \alpha\,  g(X,X) \leq h(X,X)\leq \beta \, g(X,X), \quad  \forall X\in TM.
\end{equation}
Notice that the upper bound in terms of $\beta$ immediately yields that the coefficients of $h$ in local coordinates are in $L^{\infty}$.
Let $\Omega\subset M$ be an open subset with $\meas_{g}(\Omega)\in (0,1)$ and consider the bilinear form $\Eform_{h}: W^{1,2}_{0}(\Omega) \times W^{1,2}_{0}(\Omega) \to \R$ defined by
\begin{equation}\label{eq:defEh}
\Eform_{h} (u,v) \doteq \int_{\Omega} h(\nabla u, \nabla v)\, \de \meas_{g}, \quad \forall u,v\in  W^{1,2}_{0}(\Omega).
\end{equation}
Let $A_{h}\in \Gamma_{meas}( TM{\otimes} T^{*}M)$  be the symmetric endomorphism associated to $h$ (i.e. with local representation $A_{h}=g^{-1}h$) and define the differential operator $\Delta_{h}: W^{1,2}_{0}(\Omega)\to  W^{-1,2}(\Omega)$
\begin{equation}\label{eq:defDeltah}
\Delta_{h} u\doteq {\rm div}_{g} (A_{h} \nabla u),\quad   \forall u\in  W^{1,2}_{0}(\Omega),
\end{equation}
where $ {\rm div}_{g}$ is the divergence with respect to $g$. Of course, if $h=g$ we have that $\Delta_{g}$ is the standard Laplace-Beltrami operator of $g$.
Integration by parts  gives
\begin{equation*}
\Eform_{h} (u,v) \doteq \int_{\Omega} h(\nabla u, \nabla v)\, \de \meas_{g}=\int_{\Omega} (\Delta_{h} u)\, v \, \de \meas_{g}, \quad \forall u,v\in W^{1,2}_{0}(\Omega).
\end{equation*}
Given $f\in L^2(\Omega)$, we say that a function $u\in W^{1,2}_0(\Omega)$ is a weak solution to the Poisson problem
	\begin{equation}\label{eq:defPoissonApp}
	\begin{cases}
		-\Delta_{h} u=f 	& \text{in $\Omega$}\\
		u=0						& \text{on $\partial \Omega$}
	\end{cases}
	\end{equation} 
if 
	\begin{equation}\label{eq:defPoissonAppWeak}
	\Eform_{h}(u,v)=\int_{\Omega} f v   \, \de{\rm vol}_{g}, 	\qquad \text{$\forall v\in W^{1,2}_0 (\Omega)$.}
	\end{equation}

Let $\mathbb{S}_{K}^{N}$ be a sphere of dimension $N\geq 2$ and constant Ricci curvature $K>0$. Denote with $\meas_{K,N}$ the normalized volume measure on $\mathbb{S}_{K}^{N}$.
Fix $p\in \mathbb{S}_{K}^{N}$ once for all. For every measurable subset $\Omega\subset M$ with volume $\meas_{g}(\Omega)=v\in (0,1)$, let $r_{v}>0$ be such that $\meas_{K,N}(B_{r_{v}}(p))=v$.

 For a measurable function $u:\Omega\to \R$,  define the Schwarz symmetrization $u^{\star}: B_{r_{v}}(p)\to [0,\infty]$ as 
 $$
 u^{\star}(x)\doteq u^{\sharp}(\meas_{K,N}(B_{\dist(p,x)}(p)), \quad \forall x\in B_{r_{v}}(p),
 $$
 where  $u^{\sharp}:[0, \meas(\Omega)]\to [0,\infty]$ is the decreasing rearrangement of $u$ defined in \eqref{eq:Defusharp}. Let us stress that while $u$ is defined on $\Omega\subset M$, the symmetrized function $u^{\star}$ is defined on $B_{r_{v}}(p)\subset \mathbb{S}_{K}^{N}$.
 
 Our Talenti-type comparison theorem compares the Schwarz symmetrization $u^{\star}$ of a weak solution $u$ to the Poisson problem \eqref{eq:defPoissonApp} with the weak solution $w\in W^{1,2}_{0}(B_{r_{v}}(p))$ of the following symmetrized  Poisson problem on $\mathbb{S}_{K}^{N}$:
 
 \begin{equation}\label{eq:defPoissonAppSphere}
	\begin{cases}
		-\alpha \Delta_{\mathbb{S}_{K}^{N}} w= f^{\star}	& \text{in $B_{r_{v}}(p)$}\\
		w=0						& \text{on $\partial B_{r_{v}}(p)$}
	\end{cases},
	\end{equation} 
where $\Delta_{\mathbb{S}_{K}^{N}}$ is the standard Laplace-Beltrami operator on $\mathbb{S}_{K}^{N}$. Notice that, by equi-measurability,  $f^{\star}\in L^{2}(B_{r_{v}}(p))$ with $\|f\|_{L^{2}(\Omega)}=\|f^{\star}\|_{L^{2}(B_{r_{v}}(p))}$. Moreover, since $f^{\star}$ depends only on the radial coordinate from $p$, \eqref{eq:defPoissonAppSphere} reduces to an ODE on the interval $[0,r_{v}]$, corresponding to the  model problem studied in \cref{subsec:equation_on_model_space}.

We are now in position to state our main results (\cref{theorem:talenti,thm:talenti_rigidity,thm:stabilityTalenti,Cor:SobolevIneq}) in the smooth framework.

\begin{theorem}[A Talenti-type comparison for Riemannian manifolds with positive Ricci curvature]\label{theorem:talentiAppendix}
Let $(M,g)$ be an $N$-dimensional compact Riemannian manifold without boundary with ${\rm Ric}_{g}\geq K \,g$ for some constant $K>0$, $N\geq 2$.
Let $\Omega\subset M$ be an open subset with $\meas_{g}(\Omega)=v\in (0,1)$. Let   $f\in L^2(\Omega,\meas)$ and $u\in W^{1,2}_{0}(\Omega)$ be a weak solution to the Poisson problem  \eqref{eq:defPoissonApp}, where the operator $\Delta_{h}$  was defined in \eqref{eq:defDeltah} (see also \eqref{eq:Aalphabeta}, \eqref{eq:defEh}, \eqref{eq:defPoissonAppWeak}).

Let   $w\in W^{1,2}_{0}(B_{r_{v}}(p))$ be a weak solution to the Poisson problem \eqref{eq:defPoissonAppSphere} on $B_{r_{v}}(p)\subset \mathbb{S}_{K}^{N}$. Then
\begin{itemize}
\item \emph{Pointwise comparison:}  $u^\star(x)\leq w(x)$,   for every $x\in B_{r_{v}}(p)$.
\item  \emph{Gradient comparison:}  For any $1\leq q \leq 2$, the following $L^{q}$-gradient estimate holds: 
	\begin{equation*}
	\int_\Omega |\nabla u|^q \dmeas_{g}	\leq \int_{B_{r_{v}}(p)} |\nabla w|^{q} \, \dmKN.
	\end{equation*}
\item \emph{Rigidity:} if  $u^\star(x)= w(x)$,   for at least a point  $x\in B_{r_{v}}(p)$, then  $(M,g)$ is isometric to  $\mathbb{S}_{K}^{N}$.
\item \emph{Stability:} see \cref{thm:stabilityTalenti}.
\item \emph{Improved Sobolev embeddings:} Assume that $u\in W^{1,2}_{0}(\Omega)$ satisfies $\Delta_{h} u\in L^{p}(\Omega)$;
\begin{itemize}
\item   If $\frac{N}{2}\leq p \leq \infty$, then $u\in  L^{\infty}(\Omega)$ with
		\begin{equation*}
		\norm{u}_{L^{\infty}(\Omega)}\leq C(K,N,v,p,\alpha) \norm{\Delta_{h} u}_{L^p(\Omega)}.
		\end{equation*}
	\item If $2\leq p \leq \frac{N}{2}$ and $1\leq q < \pths*{\frac{1}{p}-\frac{2}{N}}^{-1}$, then $u\in L^{q}(\Omega)$ with
		\begin{equation*}
		\norm{u}_{L^{q}(\Omega)}\leq C(K,N,v,p,q,\alpha) \norm{\Delta_{h} u}_{L^p(\Omega)}.
		\end{equation*}	

\end{itemize}
\end{itemize}
\end{theorem}

We only comment the rigidity statement: from \cref{thm:talenti_rigidity} we know that $(M,g)$ is a $(K,N)$-spherical suspension, in particular it has diameter $\pi \sqrt{\frac{N-1}{K}}$. By the Cheng's maximal diameter theorem, it follows that $(M,g)$ is isometric to the round sphere $\mathbb{S}_{K}^{N}$.


\subsection{A probabilistic interpretation in the smooth setting:  the exit time of Brownian motion}
In this section we consider the smooth setting of a compact $2\leq N$-dimensional Riemannian manifold $(M,g)$ without boundary, with Ricci curvature tensor satisfying ${\rm Ric}_{g}\geq K \,g$ for some constant $K>0$. 
Let $\Omega\subset M$ be an open subset, fix $x\in \Omega$ and let $(X_{t})_{t\geq 0}$ be the Brownian motion starting from $x$ (a good reference for the Brownian motion on Riemannian manifolds is the monograph by Hsu \cite{Hsu}).
The \emph{exit time from $\Omega$} is the Random variable $\tau_{\Omega}$ defined on the Brownian probability space as:
\begin{equation}\label{eq:defET}
\tau_{\Omega}(X):=\inf\{t>0\, :\,  X_{t}\notin \Omega \}.
\end{equation}
The connection between exit time and Poisson equations in the Euclidean setting is classical (it goes back at least to Kakutani \cite{Kakutani1944}), in the next proposition we give the natural generalization to the Riemannian setting.

\begin{proposition}\label{prop:eET}
Let $(M,g)$ be a compact Riemannian manifold without boundary and let $\Omega\subset M$ be an open subset with smooth boundary. Let $f\in C^{\infty}(\Omega)$ be a bounded smooth function.
Suppose $u\in C^{\infty}(\bar{\Omega})$ solves the Poisson problem
\begin{equation}\label{eq:PoissET}
	\left\lbrace
	\begin{aligned}
	 -  \frac{1}{2}\Delta u & = f \quad \text{in $\Omega$}\\
	 u&=0   \quad \text{on $\partial \Omega$}
	\end{aligned}
	\right. 
	\end{equation}
	Then $u$ can be written as
	\begin{equation}\label{eq:uET}
	u(x)={\mathbb E}_{x} \left(  \int_{0} ^{\tau_{\Omega}(X)} f(X_{t}) \, \de t\right),
	\end{equation}
	where $ {\mathbb E}_{x} (Y)$ denotes the mean of the random variable $Y$ on the Brownian probability space, when the Brownian motion starts at $x$.  
\end{proposition}

\begin{proof}
The proof in the Riemannian setting goes along the same lines of the classical Euclidean proof, thus we only sketch the main steps.
\\By It\^o's formula (see for instance \cite[Proposition 3.2.1]{Hsu}), the quantity 
$$
u(X_{t})-u(X_{0})-\frac{1}{2}\int_{0}^{t} \Delta u(X_{s}) \, \de s, \quad \forall t\in [0, \tau_{\Omega}(X)),
$$
defines a local Martingale. Since $u\in C^{\infty}(\bar{\Omega})$ solves \eqref{eq:PoissET}, it follows that also
$$
M_{t}\doteq u(X_{t}) + \int_{0}^{t}  f(X_{s})\, \de s
$$
is a local Martingale on $ [0, \tau_{\Omega}(X))$. Moreover,  by the very definition  \eqref{eq:defET} of exit time we have that 
$\lim_{t\uparrow \tau_{\Omega}(X)} u(X_{t})=0$. Letting $t\uparrow \tau_{\Omega}(X)$ and applying ${\mathbb E}_{x}$  gives  \eqref{eq:uET} (for more details see  for instance \cite[pp.251-252]{Durrett}).
\end{proof}

Specializing \cref{prop:eET} to $f\equiv 1$, the solution $u$ to \eqref{eq:PoissET} is then the stress function of $\Omega$ and \cref{eq:uET} gives that
$$
u(x)= {\mathbb E}_{x}(\tau_{\Omega}(X)), \quad \forall x\in \Omega.
$$
In other terms, the stress function $u$ evaluated at $x$ corresponds to the expected exit time of the Brownian motion starting at $x$. In the same spirit, 
$$
{\mathcal T}(\Omega)\doteq \frac{1}{{\rm vol}_{g}(\Omega)} \int_{\Omega} u \, \de {\rm vol}_{g}=  \frac{1}{{\rm vol}_{g}(\Omega)} \int_{\Omega} {\mathbb E}_{x}(\tau_{\Omega}(X)) \, \de {\rm vol}_{g},
$$
can be interpreted as the average exit time of a Brownian motion starting at a random point $x\in \Omega$ (with respect to the uniform probability ${\rm vol}_{g}(\Omega)^{-1} \, {\rm vol}_{g}$).
Notice that  
$${\mathcal T}(\Omega)= \frac{1}{{\rm vol}_{g}(\Omega)} T(\Omega),$$
where  $T(\Omega)$ is the torsional rigidity of $\Omega$ defined in \cref{Sec:Torsion}.
\\Hence, specializing \cref{theorem:talentiAppendix} (see also \cref{thm:StVenPol}) to $f\equiv 2$ and $h=g$ gives the following:

\begin{corollary}[Exit-time comparison]\label{cor:ET}
Let $(M,g)$ be an $N$-dimensional compact Riemannian manifold without boundary with ${\rm Ric}_{g}\geq K \,g$ for some constant $K>0$, $N\geq 2$.
Let $\Omega\subset M$ be an open subset with $\meas_{g}(\Omega)=v\in (0,1)$ and let $B_{r_{v}}\subset {\mathbb S}^{N}_{K}$  be a metric ball with $\meas_{K,N}(B_{r_{v}})=v$. Then
\begin{itemize}
\item \emph{Expected exit time comparison}:  $ \left({\mathbb E}_{(\cdot)}(\tau_{\Omega}) \right)^{\star} (x) \leq {\mathbb E}_{x}(\tau_{B_{r_{v}}})$, for all $x\in B_{r_{v}}$.
\item \emph{Average exit time comparison with rigidity}:  ${\mathcal T}(\Omega)\leq {\mathcal T}(B_{r_{v}})$. Equality holds if and only if $(M,g)$ is isometric to ${\mathbb S}^{N}_{K}$.
\item \emph{Stability}: Given $N\geq 2, v\in (0,1), \varepsilon>0$ there exists $\delta=\delta(N,v,\varepsilon)>0$ such that if $K=N-1$ and $\Omega=B_{R}(x)$ is an open metric ball with $\meas_{g}(B_{R}(x))=v\in (0,1)$ satisfying ${\mathcal T}(\Omega)\geq {\mathcal T}(B_{r_{v}})-\delta$ then $(M,g)$ is $\varepsilon$-mGH close to a spherical suspension, i.e. there exists a spherical suspension $\Zdm$ such that 
	$
	\dmGH((M,g),\Zdm)<\epsilon.
	$
\end{itemize}
\end{corollary}

Corollary \ref{cor:ET} should be compared with \cite{CLM} where also higher order average exit times are estimated in a smooth Riemannian manifold with strictly positive Ricci curvature. The novelty of Corollary \ref{cor:ET} lies in particular in the stability statement.


\nocite{Ambrosio,
		Cavalletti2017a,
		Kesavan2006,
		Mondino2019,
		Sturm2006a%
		}
		
\emergencystretch=1em

	{\hbadness=2000		
	{\small {\footnotesize
\printbibliography}}
	}					
\end{document}